\documentclass[a4paper]{article} 
\usepackage[english]{babel}

\usepackage[utf8]{inputenc}
\usepackage{lmodern}

\usepackage{amsmath,amsfonts,amssymb,amscd, amsthm}
\usepackage{authblk}

\usepackage{booktabs}
\usepackage{multirow}
\usepackage{alltt}
\usepackage{algorithm}
\usepackage{algorithmic}

\usepackage{color}
\usepackage[pdftex]{graphicx}
\usepackage[pdftex]{thumbpdf}
\graphicspath{{figures/}}

\usepackage{tikz}
\usetikzlibrary{arrows,calc,decorations.markings,intersections,positioning,shapes.arrows}

\usepackage{caption}
\usepackage{subcaption}
\usepackage[colorlinks=true]{hyperref}

\usepackage{url}

\usepackage[backend=bibtex]{biblatex}
\bibliography{main}

\newcommand{\R}{\mathbb{R}}

\newcommand{\half}{\frac{1}{2}}

\newcommand{\norm}[1]{\left\|#1\right\|}

\newcommand{\Ito}{It\^{o}}

\usepackage{pgfplots}
\usepackage{pgfplotstable}
\pgfplotsset{compat=newest}
\usepackage{booktabs}

\newcommand{\revbm}[1]{{\color{black}{#1}}}

\newcommand{\orcid}[1]{\href{https://orcid.org/#1}{\small [#1]}}
\newtheorem{proposition}{Proposition}
\newtheorem{theorem}{Theorem}
\theoremstyle{remark}
\newtheorem{remark}{Remark}

\begin{document}

\title{Toward Spectral Deferred Correction Methods for Stochastic Differential Equations}
\author[1]{Lisa Binkowski \orcid{0000-0001-5673-1388}}
\author[2]{Matthew R. Christian \orcid{0009-0003-7649-4158}}
\author[1,3]{Sebastian Götschel \orcid{0000-0003-0287-2120} }
\author[2]{Michael L. Minion \orcid{0000-0002-0044-9778}}

\affil[1]{Zuse Institute Berlin, Germany Germany, \texttt{lbinkowski@mailbox.org, goetschel@zib.de}}
\affil[2]{formerly of  Department of Mathematics, University of North Carolina at Chapel Hill, USA, \texttt{mrxian85@gmail.com, mlminion@gmail.com}}
\affil[3]{Chair Computational Mathematics, Institute of Mathematics, Hamburg University of Technology, 21073 Hamburg, Germany, \texttt{sebastian.goetschel@tuhh.de}}

\maketitle

\begin{abstract}
The numerical solution of stochastic differential equations is computationally very demanding. Often it is of interest to evaluate certain moments, e.g., the expectation, by Monte Carlo sampling, thus requiring the computation of numerous realizations. The discretization error of the numerical scheme induces bias in the computed Monte Carlo estimator. As popular numerical schemes, like Euler-Maruyama or Milstein, are low-order, even reasonably coarse error bounds for the simulation are costly to achieve. While the evaluation of moments requires only weak convergence, in other applications like parameter estimation, strong convergence is important, but higher-order is typically harder to obtain in the strong than in the weak sense. In this paper we develop an efficient method with the potential of achieving higher strong and weak convergence order. It is based on the smooth approximation of Brownian bridges combined with spectral deferred correction methods well known for the numerical solution of ordinary differential equations. We discuss convergence, and present numerical results.

\textbf{Keywords:} spectral deferred correction, stochastic differential equations, Wong-Zakai approximation, smooth Brownian bridge
\end{abstract}

\section{Introduction}

Stochastic differential equations (SDEs) are used for modeling various real-world problems in different application areas, ranging from finance, e.g., option pricing, to biology, e.g., gene expression. As most SDEs cannot be solved analytically, numerical methods are required to compute the solution. 
While in some applications it is of interest to evaluate certain moments of the solution, like the expected value, by Monte Carlo sampling, there are cases where strong solutions, i.e., trajectories are required. This includes for example systems where the solution exhibits a number of stable steady states with switching occurring due to noise~\cite{rackauckas2018stability}, or gene regulation networks, where the behavior of just one molecule can be significant~\cite{BurrageBurrageTian2004}.

The most common numerical methods for SDEs are the Euler-Maruyama and Milstein schemes that can achieve strong order one half and one, respectively, under suitable regularity assumptions like global Lipschitz-continuity of drift and diffusion terms \cite{KloedenPlaten1992}. 
As the standard Euler–Maruyama and Milstein scheme can fail to converge for SDEs with non-globally Lipschitz coefficients, modified methods have been developed for this case~\cite{HutzenthalerEtAl2012,WangGan2013,KellyLord2022}.
If such low-order methods are used to approximate SDE solutions, very fine time discretizations might be required to achieve a prescribed approximation accuracy. Also, applications like Bayesian inference for parameter estimation might benefit from higher convergence orders, see, e.g.~\cite{lie2018random}, where the \emph{strong} convergence order of the numerical methods directly translates into convergence order for the approximation of the posterior distribution.

Various attempt to achieve somewhat higher strong convergence orders have been made. They include stochastic Runge-Kutta (SRK) based methods, which are investigated in, e.g., \cite{BurrageBurrage1996,burrage1999high}, with strong convergence order up to $1.5$. SRK methods for \Ito-SDEs converging in a weak sense with order $2$ for scalar noise and with order greater than $2$ for small noise are developed in \cite{rossler2006runge}; for small noise, convergence orders up to $2$ are achieved in~\cite{buckwar2010stochastic}. 
Various other second-order SRK  methods have been developed, e.g., \cite{rosslerSecondOrderRunge2007,DEBRABANT2008408,Roessler2010}, but higher convergence orders are hard to achieve due to a large number of order conditions even for cases with commutative noise~\cite{Roessler2004}. We also refer to~\cite{BurrageBurrageTian2004,HighamKloeden2021} for an overview.

In this paper we make use of the Wong-Zakai approximation~\cite{WongZakai1965,Sussmann1978} of the driving Wiener process, replacing it with its truncated Fourier series, to transform the SDE into an ordinary differential equation (ODE) with random right-hand-side, which then is solved using an adapted variant of spectral deferred correction methods~\cite{dutt2000spectral} to obtain higher convergence orders. While the Wong-Zakai approximation is long known, only a few numerical schemes based on this approximation exist. Cao et al.~\cite{CaoEtAl2015} derive numerical methods for stochastic delay differential equations and obtain convergence rates up to $1$. Wong-Zakai approximation is used in~\cite{kamrani2018implicit} to efficiently evaluate double stochastic integrals in an implicit Milstein method, and show that this retains stability region and convergence order of the Milstein scheme. Londo\~no and Villegas~\cite{LondonoVillegas2016} combine it with a standard ODE solver to obtain numerical solutions, but do not attempt to derive higher-order schemes. 

The remainder of the paper is structured as follows. We briefly recall SDEs and spectral deferred correction (SDC) methods in section~\ref{section:preliminaries}. In section~\ref{section:SBB} we use a smooth approximation of Brownian bridges to construct SDC methods for SDEs, and discuss convergence for the simplified case of a linear Brownian bridge approximation in section~\ref{section:ConvResults}. Numerical examples are presented in section~\ref{section:NumExp}.

\section{Preliminaries}\label{section:preliminaries}
In this section, to fix notation, we will give a short introduction to SDEs, see, e.g.,~\cite{KloedenPlaten1992,Oksendal2003} for more details. Also, we will summarize the basic idea of spectral deferred correction methods for ODEs~\cite{dutt2000spectral}.

\subsection{Stochastic differential equations}

A parameterized collection of random variables $\left\{X_t\right\}_{t\in [0, T]} \in \mathbb{R}^n$ defined on a filtered probability space $\left\{\Omega, \mathcal{F},  \mathcal{F}_t, P \right\}$ with $0\leq t \leq T$ is called a stochastic process. Here, $\mathcal{F}_t$ denotes a filtration, representing the information available up to time $t$. In the context of stochastic differential equations, the process $X_t$ is typically assumed to be adapted to this filtration. A path of $X_t$ is a function $t \mapsto X_t(\omega)$ with fixed outcome $\omega \in \Omega$ and $t \in [0, T]$. Keeping $t \in [0,T]$ fixed, we obtain the random variable $\omega \mapsto X_t(\omega)$ with outcomes $\omega \in \Omega$.

A Wiener process $W=\left\{ W_t, t \geq 0 \right\}$, also called (standard) Brownian motion, is a stochastic process with  $W(0)=0$, expectation $\mathbb{E}(W(t))=0$ and variance \mbox{$\mathbb{V}(W(t)-W(s)) = t-s$} for all $0\leq s\leq t$. The equation
\begin{equation}
	\label{eq:SDE}
	dX_t = a(t,X_t)dt + b(t,X_t) dW_t
\end{equation} 
defines a scalar stochastic differential equation (SDE), with drift $a(t,X_t)$  and diffusion term $b(t,X_t)$. The path $W_t$ of a Brownian motion is continuous, but not differentiable. 
In order to give a unique meaning to the SDE, two integral forms can be defined, the \Ito \ integral form 
\begin{equation}
	\label{eq:ItoSDE}
	X_t = X_0 + \int\limits_{0}^{t} a(s,X_s) ds + \int\limits_{0}^{t} b(s,X_s) dW_s
\end{equation}
and the Stratonovich integral form 
\begin{equation}
	\label{eq:StratSDE}
	X_t = X_0 + \int\limits_{0}^{t} \underline{a}(s,X_s) ds + \int\limits_{0}^{t} b(s,X_s) \circ dW_s.
\end{equation}
%
Given a scalar \Ito-SDE~\eqref{eq:ItoSDE}, it can be transformed to Stratonovich form using the modified drift term
$\underline{a}(t,X_t) = a(t,X_t) - \half b(t,X_t)\frac{\partial}{\partial x}b(t,X_t),$
see~\cite{evans2012introduction,KloedenPlaten1992}, where also the extension to the multidimensional case can be found.

For the numerical solution of SDEs, two notions of convergence are considered here. The time discrete approximation $\hat{X}$ given step size $\Delta T > 0$ \emph{converges strongly with strong oder $p$} to the solution $X$ of equation \eqref{eq:SDE} at final time $T$ if 
for a constant $C$ the inequality 
$\mathbb{E}\bigl( \bigl|X_T -\hat{X}(T)\bigr| \bigr) \leq C \Delta T^p$
holds for all sufficiently small $\Delta T$~\cite{KloedenPlaten1992}. 
The discrete time approximation $\hat{X}$ \emph{converges weakly with weak order $q$} to the solution $X$, if
$\bigl|\mathbb{E}\bigl( g\bigl(X_T\bigr)\bigr) - \mathbb{E}\bigl( g\bigl(\hat{X}(T) \bigr)\bigr) \bigr| \leq C \Delta T^q $
holds for a class of test functions $g: \mathbb{R}^{n} \rightarrow \mathbb{R}$. If the class of test functions contains all polynomials, we get weak convergence of all moments~\cite{KloedenPlaten1992}.  We do not consider path-wise convergence, i.e.,
\[
\sup_{0\leq t\leq T}|\hat{X}(t,\omega)-X(t,\omega)|\to 0 \text{ as } \Delta T \to 0,
\]
as, for example, discussed by Jentzen et al.~\cite{JentzenEtAl2009} or recently by Allen et al.~\cite{AllenEtAl2025}.

\subsection{Spectral Deferred Corrections for ODEs}
Spectral deferred correction (SDC) methods introduced by Dutt et al.~\cite{dutt2000spectral}, and further developed, amongst others, by Minion~\cite{minion2003semi}, Weiser~\cite{weiser2015}, and Speck et al.~\cite{speck2015multi} are iterative high-order time-stepping schemes for ODEs
$	\frac{d}{dt}x(t) = a(t, x(t)), x(0) = x_0.$
Considering a single time interval $[0, \Delta T]$, the problem is written in Picard integral form as
\begin{equation}
	x(\Delta T) = x_0 + \int_{0}^{\Delta T} a(s,x(s))\, ds.
\end{equation}
Given $M+1$ collocation nodes, $0 \leq \tau_0 < \dots < \tau_M \leq \Delta T$, 
we get a similar integral form for the solution values $x_m = x(\tau_m)$ at the collocation nodes.
Starting from a low-order predictor, the initial approximation $x_m^{0}, m=0,\dots,M$, the iterates are updated to $x_m^{k+1} = x_m + \delta^k_m$ with
\begin{equation}
	\begin{split}
		\delta^k(\tau_{m+1}) &=  \delta^k(\tau_{m}) + \int_{\tau_m}^{\tau_{m+1}} a(s, x^k(s)+\delta^k(s)) -a(s, x^k(s)) ds \\ 
		&+ \int_{\tau_{m}}^{\tau_{m+1}} a(s, x^k(s)) ds + x^k(\tau_{m}) -x^k(\tau_{m+1}).
		\label{eq:Correction}
	\end{split}
\end{equation} 
Here, the first integral may, in the non-stiff case, be approximated by a left-looking rectangular rule (i.e., explicit Euler), and the second integral is approximated to high order using spectral quadrature.

\section{SDC with Smooth Brownian Bridges} 
\label{section:SBB}

To apply SDC methods to SDEs, we consider the Stratonivic integral form \eqref{eq:StratSDE}, and, to obtain higher-order methods, require information on the Wiener process between time steps.
For this, we will make use of a \emph{Brownian bridge}, i.e., a modified Wiener process for  which the sample paths all pass through given points {$(t_0, Y_0), (t_0+\Delta T, Y_1)$. Such a process $B^{\Delta T,Y_1}_{t_0,Y_0}(t,\omega)$ is defined sample path-wise for $t \in [t_0,t_0+\Delta T]$ by
	\begin{equation}\label{eq:bbridge}
		B(t) := B^{t_0+\Delta T,Y_1}_{t_0,Y_0}(t,\omega) = Y_0 + W(t\revbm{-t_0},\omega) - \frac{t -t_0}{\Delta T}(W(\revbm{\Delta T},\omega) - Y_1 + Y_0).
	\end{equation}
	A Brownian bridge is a Gaussian process with mean $\mu(t)$ and  covariance $C(s,t)$ given as
	\begin{equation}
		\mu(t) = Y_0 - \frac{t-t_0}{\Delta T}(Y_0-Y_1),\qquad C(s-t_0,t-t_0) = \min\{s-t_0,t-t_0\}  - \frac{(s-t_0)(t-t_0)}{\Delta T},
	\end{equation}
	see~\cite{KloedenPlaten1992}.
	Using a Brownian bridge makes it possible to sample values of Wiener processes between the Brownian nodes, i.e., the nodes of the macro time grid $\{t_n, n=0,\dots,N\}$, with the correct probability distribution. However, in order to use a differential equations solver with this Brownian bridge it would be necessary to make the bridge line up with the substep values that the solver needs to sample from. Also, if at any point the solver needs to differentiate the bridge, its discrete nature will be problematic. Strong numerical methods for solving SDEs avoid this issue by using integrals of $W(t)$ to achieve higher-order accuracy, which are  often computationally expensive to evaluate, e.g., in the case of vectorial equations with multidimensional Wiener processes.
	
	In this section we will use smooth approximations of Brownian bridges to construct spectral deferred correction methods (SDC-SBB) for the numerical solution of Stratonovich-SDEs.
	
	\subsection{Smooth approximation of a Brownian bridge}
	\label{subsec:SBB}
	For ease of exposition, consider a scalar SDE with a scalar standard Wiener process $W(t)$, for the multidimensional case see section~\ref{subsec:Multidim}.
	
	The basic idea is to construct a Brownian bridge between two samples from the Wiener process, and further approximate the Brownian bridge $B(t)$ by a continuously $t$-differentiable stochastic process $B^{(J)}(t)$ such that $B^{(J)}(t) \rightarrow B(t)$ as $J\rightarrow \infty$, similar to the Wong-Zakai approximation of Brownian motion. 
	In order to construct such a smooth Brownian bridge (SBB) approximation, we use the representation of Wiener processes $W(t)$ on the interval $[0,\Delta T]$ by the Karhunen-Loéve expansion~\cite{Knight1981}, 
	\begin{equation}\label{eq:WtWZ}
		W(t) = \frac{t}{\Delta T}W(\Delta T) + \frac{\sqrt{2}}{\pi} \sum_{j=1}^{\infty} \frac{\sin(j\pi t/\Delta T)}{j} \eta_j,
	\end{equation}
	with $\eta_j \in \mathcal{N}(0,1)$. Instead of directly approximating $W(t)$ by using only the linear term in eq.~\eqref{eq:WtWZ} as in the classical Wong-Zakai approximation, we aim for a local approximation of a Brownian bridge.
	For this, we proceed by inserting eq.~\eqref{eq:WtWZ} into eq.~\eqref{eq:bbridge} and truncating the summation at a finite index $J$, an approximated Brownian bridge connecting $(t_0, Y_0), (t_0+\Delta T, Y_1)$ is given as
	\begin{equation}\label{eq:SBBconstruction}
		B^{(J)}(t) = Y_0 + \frac{t-t_0}{\Delta T}(Y_1-Y_0) + \frac{\sqrt{2\Delta T}}{\pi}\sum_{j=1}^{J} \frac{\sin(j\pi {(t-t_0)}/\Delta T)}{j} \eta_j
	\end{equation}
	with $Y_0 = W(t_0), Y_1 = W(t_0+\Delta T)$  denoting the respective values from the Wiener process. Their difference thus is a normally distributed random variable with variance $\sqrt{\Delta T}$, $(Y_1-Y_0) \sim \mathcal{N}(0,\Delta T)$. 
	As we consider the interval $[0, \Delta T]$, we have $Y_0 = 0$, thus the bridge simplifies to
	\begin{equation}
		B^{(J)}(t) = \frac{t\sqrt{\Delta T}}{\Delta T}\eta_0 + \frac{\sqrt{2\Delta T}}{\pi}\sum_{j=1}^{J} \frac{\sin(j\pi t/\Delta T)}{j} \eta_j,
	\end{equation}
	with $\eta_j \in \mathcal{N}(0,1), j=0,\dots,J$.
	Note that for the numerical solution of SDEs, the time interval $[t_n,t_{n+1}]$ being solved over can always be mapped to $[0, \Delta T]$, with $\Delta T = t_{n+1}-t_n$ such that the bridge construction for that interval is no restriction.
	This smooth approximation of the Brownian bridge is similar to the piecewise truncated spectral expansion of Brownian motion~\cite{ZhangKarniadakis2017}.

	\subsection{Transforming the SDE to a random ODE}
	
	Consider again the It\^{o}-SDE
	\begin{equation}\label{eq:scalar_Ito_SDE}
		dX_t = a(t,X_t) + b(t,X_t) dW_t, \quad X(0) = x_0,
	\end{equation}
	without loss of generality for $t\in [0,\Delta T]$, i.e., we consider only one timestep as motivated above.
	Inserting the smooth Browian bridge, let $X^{(J)}(t)$  be the solution of the differential equation
	\begin{equation}\label{eq:scalar_Ito_SDE_SBB}
		\frac{dX}{dt} = a(t,X) + b(t,X)\frac{dB^{(J)}}{dt}
	\end{equation}
	corresponding to the SDE~\eqref{eq:scalar_Ito_SDE}. As the Brownian bridge $B(t)$ is equivalent to the collection of all possible Wiener processes that pass through the sampled points, and $B^{(J)}(t) \rightarrow B(t)$ for $J\rightarrow \infty$, following~\cite{Oksendal2003} the solution $X^{(J)}(t)$ converges to an $X(t)$  for $J\rightarrow \infty$ uniformly in $t$ in bounded intervals for almost all realizations $\omega \in \Omega$ .
	It is long known~\cite{WongZakai1965,Sussmann1978} that this $X(t)$ coincides with the solution of the Stratonovich formulation of the SDE~\eqref{eq:ItoSDE}, i.e., the solution of eq.~\eqref{eq:StratSDE}. Thus, when seeking a solution to an It\^{o}-SDE, it has to be re-written in Stratonovich form in order to use this approximation.
	After this transformation, and using
	\begin{equation}
		\frac{dB^{(J)}}{dt} = \frac{\eta_0}{\sqrt{\Delta T}} + \frac{\sqrt{2\Delta T}}{\Delta T}\sum_{j=1}^{J}\cos\bigl(\frac{j\pi t}{\Delta T}\bigr)\eta_j,
	\end{equation}
	the resulting random ordinary differential equation (SBB-ODE) on the interval $[0,\Delta T]$ is
	\begin{align}\label{eq:SBB_ODE}
		\frac{dx}{dt} = a(t,x) - \frac{1}{2}b(t,x)\frac{\partial}{\partial x}b(t,x) + b(t,x)\biggl(\frac{\eta_0}{\sqrt{\Delta T}} + \frac{\sqrt{2}}{\sqrt{\Delta T}}\sum_{j=1}^{J}\cos\bigl(\frac{j\pi t}{\Delta T}\bigr)\eta_j\biggr)\\
		=\underline{a}(t,x) + b(t,x)\biggl(\frac{\eta_0}{\sqrt{\Delta T}} + \frac{\sqrt{2}}{\sqrt{\Delta T}}\sum_{j=1}^{J}\cos\bigl(\frac{j\pi t}{\Delta T}\bigr)\eta_j\biggr)\label{eq:SBB_ODE2}
	\end{align}
	with $x(0) = X_0$. 
	In the following, we consider only one macro time step, the interval $[0,\Delta T]$. To continue timestepping on an interval $[n\Delta T, (n+1)\Delta T]$, the obvious changes are made, i.e., using $x(n\Delta T)$ computed on the previous time step as initial value, and evaluating the coefficient functions $a, b$ at the correct times. Note that this equation still contains a stochastic component due to the random variables $\eta_j$, which are, however, fixed on the time step.

	\subsection{Initialization}
	
	The first step of SDC is to generate an approximate solution to the random ODE~\eqref{eq:SBB_ODE} at the $M + 1$ quadrature nodes $\tau_m$. 
	With SDEs, besides a constant initialization (spread of initial value to all quadrature nodes), there are two options that can be considered in determining the initial approximation. The first is to apply the Euler method to the SBB-ODE approximation, the second is to generate the initial approximation by applying the Euler-Maruyama method to the \Ito \ form of the SDE on the time steps, combined with interpolation to the intermediate quadrature nodes. 
	
	\paragraph{Euler for the SBB-ODE.} Starting from the formulation~\eqref{eq:SBB_ODE} using the smooth Brownian bridge (SBB) approximation, we apply a standard forward Euler method and obtain
	\begin{multline}
		x^0_{m+1} = x^0_m + (\tau_{m+1} - \tau_m)\biggl(a(\tau_m, x^0_m) - \half b(\tau_m, x^0_m)\frac{\partial}{\partial x}b(\tau_m, x^0_m) \\ + b(\tau_m,x^0_m)\biggl(\frac{\eta_0}{\sqrt{\Delta T}} + \frac{\sqrt{2}}{\sqrt{\Delta T}}\sum_{j=1}^{J}\cos\bigl(\frac{j\pi \tau_m}{\Delta T}\bigr)\eta_j\biggr) \biggr),\ m=0,\dots,M-1.
	\end{multline}
	
	\paragraph{Euler-Maruyama for the \Ito-SDE with interpolation.} Here, we take one step of the Euler-Maruyama method for the original \Ito-SDE~\eqref{eq:scalar_Ito_SDE} to get the value 
	\[ x^0_M = x_0 +  a(t_0, x_0) \Delta T + b(t_0, x_0) \Delta W, \]
	where $\Delta W = W(\tau_M) - W(0) \in \mathcal{N}(0,\tau_M)$. With this, the values at the intermediate quadrature nodes are obtained by linear interpolation,
	\[
	x^0_m = x_0 + \biggl(a(0,x_0) + b(0,x_0)\Delta W\biggr) \frac{\tau_m}{\tau_M},\ m=1,\dots, M-1.
	\]
	As Euler-Maruyama has strong convergence order $\tfrac{1}{2}$, this initial approximation has order $\tfrac{1}{2}$ as well. Note that we cannot use Euler-Maruyama directly to initialize $x^0_m,\ m=1,\dots M-1$, as we require an initialization for the ODE~\eqref{eq:SBB_ODE}, and $W(t)$ and $B^{(J)}(t)$ are equivalent only at $t = n\Delta T, n=0,1,\dots$, not at the quadrature nodes. A short calculation shows that this would induce an additional local error of order $\mathcal{O}(\Delta T)$ as the random variables in the smooth Brownian bridge approximation are constant on the time steps $[n\Delta T, (n+1)\Delta T]$, thus leading to a non-convergent initialization.
	
	\subsection{Correction Sweeps}
	
	Consider the integral form of the random ODE~\eqref{eq:SBB_ODE2},
	\begin{equation}
		x(t) = x_0 +\int_0^t \underline{a}(\tau,x) + b(\tau,x)\frac{d}{d\tau} B^{(J)}(\tau) \textup{d} \tau, 
	\end{equation}
	where due to the \Ito-Stratonovich correction $\underline{a}(t,x) = a(t, x) - \tfrac{1}{2} b(t, x)\frac{\partial}{\partial x}b(t, x).$
	With an approximate solution $\hat{x}(t)$, the error is given by
	\begin{equation}
		E(t,\hat{x}(t)) = x_0+\int_{0}^{t}\underline{a}(\tau,\hat{x}) + b(\tau,\hat{x})\frac{d}{d\tau} B^{(J)}(\tau) \textup{d} \tau - \hat{x}(t),
	\end{equation}
	and a correction $\delta(t) = x(t)-\hat{x}(t)$ can be computed as
	\begin{multline}
		\delta(t) = \int_{0}^{t}\underline{a}(s,\hat{x}+\delta) + b(s,\hat{x}+\delta)\frac{d}{ds} B^{(J)}(s) -\bigl(\underline{a}(s,\hat{x}) + b(s,\hat{x})\frac{d}{ds} B^{(J)}(s)\bigr) \textup{d}s \\+ E(t,\hat{x}).
	\end{multline}
	Stepping through the quadrature nodes using the explicit Euler scheme, the correction sweep $k$ can be computed as
	\begin{multline}\label{eq:correctionsweep}
		\delta^k_{m+1} = \delta^k_m + \Delta \tau_m\biggl( \underline{a}(\tau_m,\hat{x}^k_m+\delta^k_m) - \underline{a}(\tau_m,\hat{x}^k_m) + b(\tau_m,\hat{x}^k_m+\delta^k_m)\frac{d}{ds} B^{(J)}(\tau_m) \\- b(\tau_m,\hat{x}^k_m)\frac{d}{ds} B^{(J)}(\tau_m) \biggr) + E(\tau_{m+1},\hat{x}^k) - E(\tau_m, \hat{x}^k),
	\end{multline}
	where $\Delta \tau_m = \tau_{m+1}-\tau_m$, $\delta^0_0=0$, and $\hat{x}^k_m$ denotes the approximate solution at quadrature node $m$ in sweep $k$. The generalization to implicit Euler or IMEX schemes, e.g., for stiff problems, is straightforward. Let $I_m^{m+1}(\hat{x}^k) = I_D(\hat{x}^k) + I_S(\hat{x}^k)$, where $I_D, I_S$ are the spectral quadrature approximations of the deterministic and stochastic parts of the problem,
	\[
	I_D(\hat{x}^k) \approx \int_{\tau_m}^{\tau_{m+1}} \underline{a}(s,\hat{x}^k(s))\textup{d}s,\quad
	I_S(\hat{x}^k) \approx \int_{\tau_m}^{\tau_{m+1}} b(s,\hat{x}^k(s))\frac{d}{ds} B^{(J)}(s)\textup{d}s.
	\]
	With $\hat{x}^{k+1}_m = \hat{x}^k_m + \delta^k_m$ and 
	\[
	E(\tau_{m+1},\hat{x}^k) - E(\tau_{m},\hat{x}^k) = I_m^{m+1}(\hat{x}^k) - \hat{x}^k_{m+1} + \hat{x}^k_{m},
	\]
	the final correction scheme is given as
	\begin{multline} \label{corrScheme}
		\hat{x}^{k+1}_{m+1} = \hat{x}_m^{k+1} + \Delta \tau_m \biggl( \underline{a}(\tau_m,\hat{x}^{k+1}_m) - \underline{a}(\tau_m,\hat{x}^k_m) \\+ \bigl(b(\tau_m,\hat{x}^{k+1}_m)- b(\tau_m,\hat{x}^k_m)\bigr)\frac{d}{ds} B^{(J)}(\tau_m) \biggr) + I_m^{m+1}(\hat{x}^k).
	\end{multline}

	\subsection{Extension to multidimensional problems}
	\label{subsec:Multidim}
	
	Consider the vectorial \Ito-SDE with 
	$R$ independent scalar Wiener processes $W^r$, 
	\begin{equation}\label{eq:vector_Ito_SDE}
		\textup{d}\mathbf{X}_t = a(t,\mathbf{X}_t) \textup{d}t + \sum_{r=1}^R b_r(t,\mathbf{X}_t) \textup{d}W^r_t
	\end{equation}
	where $\mathbf{X} \in \R^d$, $W^r \in \R$, $a : \R \times \R^d \rightarrow \R$, and $b_r : \R \times \R^d \rightarrow \R^d$. 
	The modified drift term for the Stratonovich-SDE is given componentwise by 
	\begin{equation}
		\underline{a}^i(t,\mathbf{X}_t) = a^i(t,\mathbf{X}_t) -\half \revbm{\sum_{r=1}^R} \sum_{j=1}^{d}b_r^j(t,\mathbf{X}_t)\frac{\partial}{\partial \mathbf{X}^j}b^i_r(t,\mathbf{X}_t),\quad i=1,\dots,d.
	\end{equation}
	Assuming that the $b_r$ commute, i.e.,
	\begin{equation}
		\sum_{j=1}^d b_r^j \frac{\partial}{\partial \mathbf{X}^j}b_s(t,\mathbf{X}_t) = \sum_{j=1}^d b_s^j \frac{\partial}{\partial \mathbf{X}^j}b_r(t,\mathbf{X}_t)\quad \forall r,s=1,\dots, R,
	\end{equation}
	then again, by replacing the Wiener processes $W_r(t)$ by a smooth approximation $B^{(J)}_r(t)$ of a Brownian bridge $B_r(t)$  with $B^{(J)}_r(t) \rightarrow B_r(t)$, passing through the sampled values of the Wiener process at given timesteps, the resulting ODE approximation of the \Ito-SDE~\eqref{eq:vector_Ito_SDE} converges to a solution of the corresponding Stratonovich-SDE~\cite{Sussmann1978}. Commutativity is a usual assumption when dealing with smooth approximations of Brownian motion to show convergence to the Stratonovich-SDE: The limiting solution for $J\to\infty$ depends on the limiting Levy areas involving the commutator $[b_r, b_s] = \frac{\partial b_r(x)}{\partial x} b_s(x) - \frac{\partial b_s(x)}{\partial x} b_r(x) $, which then will occur in the SDE. This is avoided by commutative noise, for which $[b_r, b_s] =0$ \cite{Sussmann1978}.
	
	\section{Convergence Analysis}
	\label{section:ConvResults}
	
	To analyze the convergence behavior of the SDC-SBB method we have to distinguish three contributions to the approximation error. Considering a single realization of the driving Wiener process $W(t)$ and a uniform time discretization with time step $\Delta T$, we consider the error a singe time interval $[0,\Delta T]$ of the true solution $X(t)$ and its approximation $\hat{x}_m$ at quadrature nodes $\tau_m, m = 0,\dots,M$.} Error contributions are $\varepsilon^\text{SBB}$, the error coming from the smooth Brownian bridge approximation of Brownian motion, the SDC iteration error $\varepsilon^\text{SDC}$, and the time discretization error $\varepsilon^{\Delta T}$. 
	It is well known that for ODEs SDC converges to the collocation solution, if it converges at all; as the SDC iteration error can in principle be reduced to machine precision by the correction sweeps we neglect it in the following discussion. For a less cluttered notation, we do not explicity state the time-dependency of the right-hand side functions $a(t,x), b(t,x), f(t,x)$ in this section, and just write $a(x), b(x), f(x)$.
	
	\subsection{Time discretization} 
	
	For analyzing time discretization error and convergence order of SDC, we follow~\cite{CausleySeal2017}. Consider the error at the collocation nodes $\tau_m$ on the time interval $[0,\Delta T]$.
	The SDC update in sweep $k+1$ described above is, in abstract form, given as 
	\begin{equation}\label{eq:abstract_sdc}
	\begin{split}
		\hat{x}_{m}^{k+1} = \hat{x}_{m-1}^{k+1}  + \Delta \tau_{m-1}\bigl(f(\hat{x}_{m-1}^{k+1})-f(\hat{x}_{m-1}^{k})\bigr) &+ \Delta T \sum_{n=0}^{M} w_{m,n} f(\hat{x}_n^{k}),\\
		&m=1,\dots,M,
	\end{split}
	\end{equation}
	with $f$ denoting the generic right-hand side of the abstract initial value problem in Picard integral form
	\begin{equation}\label{eq:abstract_picard}
	x(t) = x_0 + \int_{0}^{t} f(x(s))\,ds, \ t>0.
	\end{equation}
	Considering collocation nodes $\tilde{\tau}_0,\dots,\tilde{\tau}_M$ on the unit interval $[0,1]$, the quadrature weights are defined by 
	\begin{equation}
	w_{m,n} = \int_{\tilde\tau_{m-1}}^{\tilde\tau_m} \ell_n(s)\,ds,\ m=1,\dots,M, n=0,\dots,M,
	\end{equation}
	where $\ell_n(t)$ is the Lagrange interpolation polynomial of degree $\leq M$ corresponding to the quadrature node $\tilde{\tau}_n$. The nodes on an arbitrary interval $[t_0,t_0+\Delta T]$ are then given as $\tau_n = t_0 + \tilde{\tau}_n\,\Delta T$.
	Following~\cite{CausleySeal2017} and subtracting the Picard equation~\eqref{eq:abstract_picard} from the SDC iteration equation~\eqref{eq:abstract_sdc}, we get the discrete error evolution 
	\begin{equation}\label{eq:err_recursion}
	e_m^{k+1} = e_{m-1}^{k+1} + \Delta \tau_{m-1}\bigl(f(\hat{x}_{m-1}^{k+1}) - f(\hat{x}_{m-1}^{k})\bigr) + \int_{\tau_{m-1}}^{\tau_m} \mathcal{I}(f(\hat{\mathbf{x}}^k))(t) - f(x(t))\, dt,
	\end{equation}
	for the error $e_m^k = \hat{x}_m^k - x(\tau_m)$, where $\hat{\mathbf{x}} = (\hat{x}_{0},\dots,\hat{x}_{M})$ denotes the vector of the approximate solution at the collocation nodes, and the superscript $k$ indicates the $k$-th sweep. The interpolation operator $\mathcal{I}$ is defined by 
	\begin{equation}
	\mathcal{I}(f(\hat{\mathbf{x}}))(t) = \sum_{n=0}^{M}f(\hat{x}_n) \ell_n(t/\Delta T),
	\end{equation}
	giving the identity  
	\begin{equation}
	\int_{\tau_{m-1}}^{\tau_m} \mathcal{I}(f(\hat{\mathbf{x}}))(t) \, dt = \Delta T\sum_{n=0}^{M}w_{m,n} f(\hat{x}_n).
	\end{equation}
	This yields the integration error 
	\begin{equation}
	\begin{split}
		\left|\int_{\tau_{m-1}}^{\tau_m} \mathcal{I}(f(\hat{\mathbf{x}}^k))(t) - f(x(t))\, dt\right| &\leq \left|\int_{\tau_{m-1}}^{\tau_m} \mathcal{I}(f(\hat{\mathbf{x}}^k))(t) - \mathcal{I}(f(\mathbf{x}))(t)\, dt\right| \\
		&+\left|\int_{\tau_{m-1}}^{\tau_m} \mathcal{I}(f(\mathbf{x}))(t) - f(x(t))\, dt\right|
		\label{eq:interp}
	\end{split}
	\end{equation}
	by adding and subtracting the Lagrange interpolant $\mathcal{I}(f(\mathbf{x}))$, where $\mathbf{x} = (x(\tau_0), \dots, x(\tau_m))$.
	For the first term we get 
	\begin{equation}\label{eq:interpol_integral_error}
	\begin{split}
		\left|\int_{\tau_{m-1}}^{\tau_m} \mathcal{I}(f(\hat{\mathbf{x}}^k))(t) - \mathcal{I}(f(\mathbf{x}))(t)\, dt\right| &=
		\left| \Delta T\sum_{n=0}^{M}w_{m,n} \bigl(f(\hat{x}_n)-f(x(\tau_n))\bigr)\right| \\
		&\leq \Delta T \sum_{n=0}^{M}|w_{m,n}| |f(\hat{x}_n)-f(x(\tau_n))|.
	\end{split}
	\end{equation}
	\begin{proposition}\label{prop:firstintegralbnd}
	Let the right-hand side $f(x)$ be given by using only the linear part of the Brownian bridge approximation, i.e., $f(x) = \underline{a}(x) + b(x) \frac{\eta_0}{\sqrt{\Delta T}}$. Assume further, that the modified drift as well as the diffusion coefficient functions $\underline{a}, b$ are Lipschitz continuous with constants $L_{\underline{a}}, L_b$. Let $\Lambda_m = \sum_{n=0}^{M}\int_{\tilde\tau_{m-1}}^{\tilde\tau_m}|\ell_n(s)|\, ds$.  Then 
	\begin{equation}
		\left|\int_{\tau_{m-1}}^{\tau_m} \mathcal{I}(f(\hat{\mathbf{x}}^k))(t) - \mathcal{I}(f({\mathbf{x}}))(t)\, dt\right| \leq (\Delta T L_{\underline{a}} + \sqrt{\Delta T} L_b |\eta_0|) \Lambda_m \max_{n} |\hat{x}_n - x(\tau_n)|.
	\end{equation}
	\end{proposition}
	
	\begin{proof}
	Using Lipschitz continuity we get
	\begin{equation}
		\begin{split}
			\left|f(\hat{x}_n)-f(x(\tau_n))\right| &= \left| \underline{a}(\hat{x}_n) -\underline{a}(x(\tau_n)) + \bigl(b(\hat{x}_n) - b(x(\tau_n))\bigr)\frac{dB^{(0)}(t)}{dt}\right| \\
			&\leq \max_{n} \left|\hat{x}_n - x(\tau_n)\right|\, (L_{\underline{a}} + L_b |\eta_0|/\sqrt{\Delta T})\label{eq:lipschitz}.
		\end{split}
	\end{equation}
	Inserting this into~\eqref{eq:interpol_integral_error}  and using $|w_{m,n}| \leq \int_{\tilde\tau_{m-1}}^{\tilde\tau_m}|\ell_n(s)|\, ds$ yields the desired estimate.
	\end{proof}

	\begin{remark}
	In contrast to~\cite[Lemma 2.1]{CausleySeal2017}, due to the Brownian bridge approximation we get the $\sqrt{\Delta T}$ term multiplying the error from the previous SDC iteration instead of $\Delta T$. Thus, SDC iterations for the SBB approximation only gain \emph{half} an order per correction iteration. Choosing an arbitrary but fixed $J>0$ does not change the estimate significantly, as it only modifies~\eqref{eq:lipschitz} to 
	\begin{equation}
		\max_{n} \left|\hat{x}_n - x(\tau_n)\right|\, (L_{\underline{a}} + L_b |\eta_0|/\sqrt{\Delta T} + \sum_{j=1}^J L_b |\eta_j| \sqrt{2}/\sqrt{\Delta T}).
	\end{equation}
	\end{remark}
	The estimate for the second integral in~\eqref{eq:interp} depends on the smoothness of $f$ and the choice of the quadrature nodes. 
	\begin{proposition}\label{prop:secondintegralbnd}
	Consider again the linear part of the SBB approximation only, $(f\circ x)(t) = \underline{a}(x(t)) + b(x(t)) \frac{\eta_0}{\sqrt{\Delta T}},$ and assume  $f\circ x \in C^{M+1}([0,\Delta T])$. Then
	\begin{equation}\label{eq:secondIntegralBnd}
		\left|\int_{\tau_{m-1}}^{\tau_m} \mathcal{I}(f(\mathbf{x}))(t) - f(x(t))\, dt\right| \leq \frac{\tilde{F}}{(M+1)!} \Delta T^{\half(M+2)}.
	\end{equation}
	\end{proposition}
	\begin{proof}
	With
	\begin{align*}
		\frac{d}{dt}(f\circ x) &=  \frac{d\underline{a}(x)}{dx} \frac{dx}{dt} + \frac{db(x)}{dx} \frac{dx}{dt}\\
		&= \frac{d\underline{a}(x)}{dx}\left(\underline{a}(x) + b(x) \frac{\eta_0}{\sqrt{\Delta T}}\right) +  \frac{db(x)}{dx}\left(\underline{a}(x) + b(x) \frac{\eta_0}{\sqrt{\Delta T}}\right)\frac{\eta_0}{\sqrt{\Delta T}},
	\end{align*}
	the $M+1$-st derivative is bounded by
	\begin{equation}
		\norm{\frac{d^{M+1}}{dt^{M+1}}(f\circ x)}_\infty \leq \frac{\tilde{F}}{(\Delta T)^{\half(M+2)}},
	\end{equation}
	with $\tilde{F}$ independent of $\Delta T$. This yields for the interpolation error
	\begin{equation}
		\begin{split}
			\left|\mathcal{I}(f(\mathbf{x}))(t) - f(x(t))\right| \leq \frac{\tilde{F}}{(M+1)!}(\Delta T)^{-\half(M+2)} \prod_{n=0}^{M} |t-\tau_n| \\
			\leq \frac{\tilde{F}}{(M+1)!}(\Delta T)^{-\half(M+2)} (\Delta T)^{M+1} = \frac{\tilde{F}}{(M+1)!}(\Delta T)^{\half M},
		\end{split}
	\end{equation}
	and integrating gives another factor $\Delta T$ and thus yields the claimed estimate.
	\end{proof}
	
	\begin{remark}
	Standard interpolation error estimates assuming $f\circ x \in C^{M+1}([0,\Delta T])$ and $\norm{\frac{d^{M+1}}{dt^{M+1}}(f\circ x)}_\infty \leq F$ would lead to the bound
	\begin{equation}
		\left|\int_{\tau_{m-1}}^{\tau_m} \mathcal{I}(f(\mathbf{x}))(t) - f(x(t))\, dt\right| \leq \frac{F}{{(M+1)}!}\Delta T^{M+{2}},
	\end{equation}
	see, e.g.,~\cite{CausleySeal2017}. 
	In contrast to this, due to the SBB approximation, the maximum order of the method is halved here.
	\end{remark}
	
	With these estimates, following closely the remaining convergence order proof of~\cite[Thm.~3.1, Corr.~3.3]{CausleySeal2017}, we arrive at the final estimate for the error of a single step of explicit SDE-SBB.
	\begin{theorem}
	Under the conditions Propositions~\ref{prop:firstintegralbnd} and~\ref{prop:secondintegralbnd}, the error for explicit SDE-SBB satisfies
	\begin{equation}
		\bigl| e_n^{k+1} \bigr| \leq  
		C_1 \sqrt{\Delta T} \| \mathbf{e}^k \| + C_2\Delta T^{\frac{1}{2}(M+2)},
	\end{equation}
	with constants $C_1, C_2$ and $\mathbf{e}^k = (e_0^k, \dots, e_{M}^k)^T.$ 
	\end{theorem}
	
	\begin{proof}
	Starting from eq.~\eqref{eq:err_recursion}, we estimate the individual terms in the error recursion
	\begin{equation}
		|e_m^{k+1}| \leq | e_{m-1}^{k+1}| + \underbrace{|\Delta \tau_{m-1}\bigl(f(\hat{x}_{m-1}^{k+1}) - f(\hat{x}_{m-1}^{k})\bigr)|}_{=: \Delta f} + \underbrace{\biggl| \int_{\tau_{m-1}}^{\tau_m} \mathcal{I}(f(\hat{\mathbf{x}}^k))(t) - f(x(t))\, dt\biggl|}_{=: I}.
	\end{equation}
	Setting $\Lambda := \max_{0\leq n\leq M} \Lambda_n$, $\|\mathbf{e}^k\| = \max_{0\leq m \leq M} |\hat{x}_m^k-x(\tau_m)|,$ and abbreviating $\mathcal{L} := \Delta T L_a+\sqrt{\Delta T}L_b|\eta_0|$, we use Propositions~\ref{prop:firstintegralbnd} and~\ref{prop:secondintegralbnd} as well as eq.~\eqref{eq:interp} and~\eqref{eq:lipschitz} to obtain
	\begin{equation}
		I  \leq \mathcal{L}\Lambda \| \mathbf{e}^k\| + \frac{\tilde F}{(M+1)!} \Delta T^{\half(M+2)} \quad \text{and}\quad
		\Delta f \leq \mathcal{L} |\hat{x}^{k+1}_{m-1}-\hat{x}^{k}_{m-1}|.
	\end{equation}
	With $\Delta \tau_{m-1} \leq \Delta T$, 
	\begin{equation}
		|\hat{x}^{k+1}_{m-1}-\hat{x}^{k}_{m-1}| =  |\hat{x}^{k+1}_{m-1} - x(\tau_{m-1}) + x(\tau_{m-1})-\hat{x}^{k}_{m-1}| \leq |e^{k+1}_{m-1}| + |e^k_{m-1}|,
	\end{equation}
	and $|e^k_{m-1}|\leq \|\mathbf{e}^k\|$, we get
	\begin{equation}
		\begin{split}
			|e_m^{k+1}| & \leq |e^{k+1}_{m-1}| + \mathcal{L}(|e^{k+1}_{m-1}|+|e^{k}_{m-1}|)+\mathcal{L}\Lambda \|\mathbf{e}^k\| + \frac{\tilde F}{(M+1)!} \Delta T^{\half(M+2)}\\
			&= (1+\mathcal{L}) |e^{k+1}_{m-1}| + \mathcal{L}|e^{k}_{m-1}|+\mathcal{L}\Lambda \|\mathbf{e}^k\| + \frac{\tilde F}{(M+1)!} \Delta T^{\half(M+2)}\\
			&\leq (1+\mathcal{L}) |e^{k+1}_{m-1}| + \mathcal{L}(1+\Lambda)\|\mathbf{e}^k\|+ \frac{\tilde F}{(M+1)!} \Delta T^{\half(M+2)}
		\end{split}
	\end{equation} 
	Now, appealing to Corollary 2.3 of Causley and Seal~\cite{CausleySeal2017} for the recursion, we get
	\begin{equation}
		|e_m^{k+1}| \leq (1+\mathcal{L}) |e_0| + m (1+\mathcal{L})^{m-1}\biggl( \mathcal{L}(1+\Lambda)\|\mathbf{e}^k\| + \frac{\tilde F}{(M+1)!} \Delta T^{\half(M+2)}\biggr).
	\end{equation}
	Assuming $|e_0| = 0$ due to the given initial condition, using $m\leq M$,  estimating $\mathcal{L} \leq C \sqrt{\Delta T}$ with a constant $C$, and setting $C_1 = C M(1+\mathcal{L})^{M-1}(1+\Lambda)$, $C_2 = M(1+\mathcal{L})^{M-1}\frac{\tilde F}{(M+1)!}$  yields the desired estimate.
	\end{proof}

	\begin{remark}\label{rem:worstcasehalforder}
	In contrast to standard SDC convergence results, here only half an order is gained per sweep, and the maximum order is halved as well.
	Using additional terms of the SBB approximation \emph{decreases} the maximum order of accuracy, as due to the $\frac{\sqrt{2}}{\sqrt{\Delta T}}\cos\bigl(\frac{j\pi t}{\Delta T}\bigr)\eta_j$ terms, the bound for the $M+1$-st derivative is
	\begin{equation}
		\norm{\frac{d^{M+1}}{dt^{M+1}}(f\circ x)}_\infty \leq \frac{\tilde{F}}{(\Delta T)^{\half(2(M+1)+1)}},
	\end{equation}
	and the bound in \eqref{eq:secondIntegralBnd} becomes  $\frac{\tilde{F}}{(M+1)!} \sqrt{\Delta T}$. Thus, in this case spectral quadrature involving Lagrange interpolation of the oscillatory terms in the SBB approximation cannot be used. Alternatives include specialized quadrature for highly oscillating integrals, e.g., Filon-type quadrature~\cite{Iserles2004,DeanoHuybrechsIserles2017}.
	\end{remark}

	\subsection{Smooth Brownian bridge approximation error}
	
	The results above cover the convergence of SDC-SBB to the true solution of the SBB-ODE. As for convergence to the SBB-ODE to the true SDE,
	there is a large amount of literature on the Wong-Zakai approximation to the true Wiener process, see, e.g., the survey~\cite{WOS:A1996UN21100003}. Convergence rates are derived, e.g., in ~\cite{Hausenblas01112024} for SDEs with multiplicative noise, or in~\cite{BenAmmouLanconelli2018}, where it is explained that for a class of SDEs driven by scalar Brownian motion, the speed of convergence of the SDE solution coincides with the speed of convergence of the smoothed Brownian motion to the original one. For this we have according to~\cite[Lemma 2.2.16]{ZhangKarniadakis2017} for a sufficiently small $\epsilon$
	\begin{equation}\label{eq:convorderWZ}
	|W(t)-W^{(J)}(t)| = \mathcal{O}\biggl(\biggl(\frac{\Delta T}{J}\biggr)^{1/2-\epsilon}\biggr),
	\end{equation}
	where $W^{(J)}$ denotes the Karhunen-Loéve expansion of Brownian motion with truncation index $J$. Using this estimate with the smooth Brownian bridge approximation \eqref{eq:SBBconstruction} in combination with \eqref{eq:bbridge} yields the same convergence order for the smooth Brownian bridge $B^{(J)}(t)$ to the true bridge, $B(t)$. 
	Eq.~\eqref{eq:convorderWZ} indicates that the number of terms in the expansion has to grow with shrinking time step size, e.g., $J=1/\Delta T$ for a temporal convergence order of $1-2\epsilon$. Additional terms in the expansion not only require specialized quadrature for the computation of $I_S$, but pose restrictions on the number of collocation nodes due to substepping by an Euler method in the computation of the correction sweep, see eq.~\eqref{eq:correctionsweep}. As a rule of thumb, for an explicit Euler method, the stepsize, i.e., here the distance between two quadrature nodes, should be less than $1/4$-th of the oscillation period. For a term involving $\cos(j\pi t/\Delta T)$ with $t\in [0,\Delta T]$ this would result in $M>2j$ collocation nodes.
	Options to avoid the need for fine discretization include introducing terms to the expansion that reduce the variance error without high-frequency oscillations, or using higher-order substeppers. A thorough analysis of these convergence issues and potential remedies is out of scope for this paper, but will be a topic for future research.

	However, there is a special case, where it is easy to show that the Wong-Zakai approximation with only the linear term is sufficient.
	The Stratonovich-SDE
	\begin{equation}
	\label{eq:ReducSDE}
	dX_t = \alpha b(X_t)\ dt + b(X_t)\circ dW_t
	\end{equation}
	has the general solution
	\begin{equation}
	X_t = h^{-1}(\alpha t + W_t + h(X_0)),
	\end{equation}
	where $h$ is given by
	\begin{equation}\label{eq:def_h}
	h(x) = \int^x_0 \frac{ds}{b(s)},
	\end{equation}
	see~\cite[Chap.~4.4]{KloedenPlaten1992}. The SBB approximation on one time step $[0,\Delta T]$
	\begin{equation}
	\frac{dx}{ds} = \alpha b(x) + b(x)\biggl(\frac{\eta_0}{\sqrt{\Delta T}} + \frac{\sqrt{2}}{\sqrt{\Delta T}} \sum_{j=1}^{J}\cos{\bigl(\frac{j\pi s}{\Delta T}\bigr)}\eta_j \biggr)
	\end{equation}
	can be solved using separation of variables,
	\begin{equation}
	\int_{x_0}^{x_{\Delta T}} \frac{dx}{b(x)} = \int_{0}^{\Delta T} \biggl( \alpha + \frac{\eta_0}{\sqrt{\Delta T}} + \frac{\sqrt{2}}{\sqrt{\Delta T}} \sum_{j=1}^{J}\cos{\bigl(\frac{j\pi s}{\Delta T}\bigr)}\eta_j\biggr)\ ds,
	\end{equation}
	yielding
	\begin{equation}
	x_{\Delta T} = h^{-1}\biggl(\alpha \Delta T + \sqrt{\Delta T}\eta_0 + \frac{\sqrt{2\Delta T}}{\pi}\sum_{j=1}^{J}\sin{\bigl( \frac{\pi j \Delta T}{\Delta T}\bigr)}\frac{\eta_j}{j} + h(x_0)\biggr)
	\end{equation}
	with $h$ defined by~\eqref{eq:def_h}. As $\sin(m\pi)=0$ and $\sqrt{\Delta T}\eta_0 \sim \mathcal{N}(0,\Delta T)$, i.e., is a realization of the Wiener process $W_t$, the solutions of the original SDE and the SBB-ODE approximation coincide at the discrete time steps. For problems of this class, the Wong-Zakai approximation is exact, even using only the linear part, and the overall approximation error thus is determined only by the approximation of the SBB-ODE solution.
	
	\begin{remark}
	The observation, that for this class of SDEs using only the linear term of the Wong-Zakai approximation is sufficient has also been made in~\cite{LondonoVillegas2016}. However, it is not used to improve numerical results in the sense that SDEs of the form~\eqref{eq:ReducSDE} can essentially be treated as ODEs, even if $b(X_t)$ is nonlinear. 
	Geometric Brownian motion, e.g., used in the Black-Scholes model for stock prices (e.g.,~\cite[Chap.~9]{IBE2013263}), is one important stochastic process satisfying this condition.
	\end{remark}

	\section{Numerical Examples} 
	\label{section:NumExp}
	
	We show experimental convergence order and computational effort of the SDC-SBB method for four test problems, and compare performance to the standard Euler-Maruyama- and Milstein schemes. The computational effort is measured by counting the right-hand-side function evaluations of the drift and diffusive term separately. 
	We consider strong and weak convergence; for weak convergence we use the approximation of the expectation value as a proxy. If not stated differently, the approximation $\hat{X}_{T}$ is compared to the exact solution $X_T$ at final time $T$. The underlying time interval of all examples is $[0,1]$. Since the SDC approximations converge to the solution of the Stratonovich-SDE, the problems are set in this form directly. All examples use only the linear Brownian bridge expansion term. The numerical experiments are implemented in MATLAB R2024b. The code is publicly available~\cite{binkowski_2026_23387536}. 
	
	\subsubsection{Linear scalar SDE.}
	We consider the linear SDE
	\begin{equation}
	X_t = \left(a-\frac{1}{2}b^2\right)X_t dt + bX_t \circ dW_t,\quad X(0) = X_0,
	\label{ex:TP1}
	\end{equation}
	with $a,b, X_0 \in \R$. 
	Its exact solution is given by
	\begin{equation}
	X_t = X_0 \exp \left( \left( a-\frac{1}{2}b^2 \right) t + bW_t \right).
	\end{equation}
	We choose the initial value $X_0 = 1$, and set $a=0.5$ and $b=1$ to focus on the approximation of the stochastic part only.
	%
	
	We choose a Lobatto time grid with 5 collocation points, $10^4$ realizations to compute the expectation, and vary the number of SDC sweeps from $0$ to $8$ or $9$ depending on the initialization used. Figure \ref{fig:OrderSweepsAllInits} shows the convergence order for increasing number of SDC sweeps.  
	The theoretical convergence orders are plotted as dashed lines, showing that they fit the experiment. 
	
	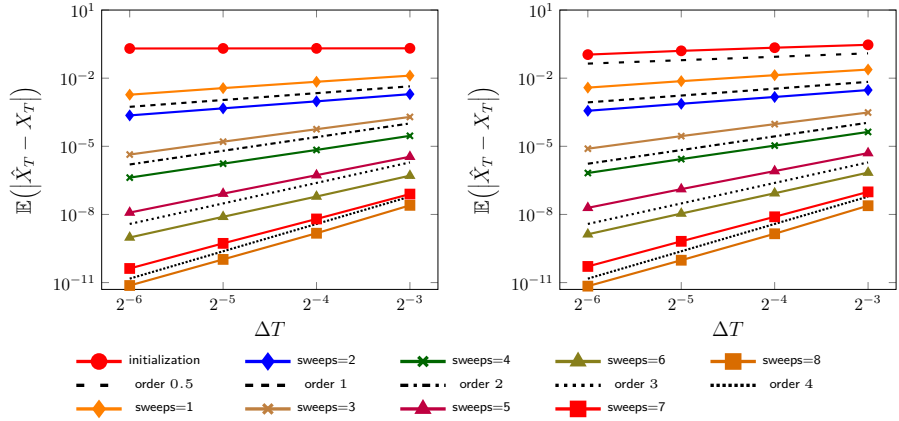
\begin{figure}[tbh!]
	\centering
	\begin{subfigure}[b]{0.49\textwidth} 
		\centering \resizebox{\linewidth}{!}{ 
			\centering	\begin{tikzpicture}[scale=0.75]
				\begin{axis}[
					every axis plot/.style={
						line width=1.25pt,
						mark size=2.5
					},
					label style={font=\sffamily},
					tick label style={font=\sffamily},
					xmode=log,
					ymode=log,
					log basis x = {2},
					xlabel={\large $\Delta T$}, 
					ylabel= \large {$ \mathbb{E} \bigl( \bigl\vert \hat{X}_{T}  - X_{T}  \bigr\vert \bigr)  $},
					ymin=0.5e-11, ymax = 10,
					xtick= data,
			legend style={draw=none, at={(0.45,-0.2)},anchor=north, font = \tiny \sffamily, /tikz/every even column/.append style={column sep=0.5cm}}, 
			legend columns=3,
			transpose legend,
			legend to name={legend:convEx1a},
			domain = 2^-6:2^-3,
			]
			\addplot[color=red,mark=*]
			table[x=stepSize,y=errStr]{data_bb_1_orderSweeps/Err_compEff_SDC_BB_exp_5_lo_eulerS_01_10000_64_0.dat};
			\addlegendentry{initialization};
			\addplot[solid,color=orange,mark=diamond*]
			table[x=stepSize,y=errStr]{data_bb_1_orderSweeps/Err_compEff_SDC_BB_exp_5_lo_eulerS_01_10000_64_1.dat};
			\addlegendentry{sweeps=1}; 
		\addplot[solid,color=blue,mark=diamond*]
		table[x=stepSize,y=errStr]{data_bb_1_orderSweeps/Err_compEff_SDC_BB_exp_5_lo_eulerS_01_10000_64_2.dat};
		\addlegendentry{sweeps=2}; 
	\addplot+[mark=none,dashed,color=black]{ 0.035*x };					\addlegendentry{order $1$};	
	\addplot[solid,color=brown,mark=x]
	table[x=stepSize,y=errStr]{data_bb_1_orderSweeps/Err_compEff_SDC_BB_exp_5_lo_eulerS_01_10000_64_3.dat};
	\addlegendentry{sweeps=3}; 
	\addplot[solid,color=green!40!black,mark=x]
	table[x=stepSize,y=errStr]{data_bb_1_orderSweeps/Err_compEff_SDC_BB_exp_5_lo_eulerS_01_10000_64_4.dat};
	\addlegendentry{sweeps=4}; 
	\addplot+[mark=none,dashdotted,color=black]{ 0.0065*x^2 };				\addlegendentry{order $2$};
	\addplot[solid,color=purple,mark=triangle*]
	table[x=stepSize,y=errStr]{data_bb_1_orderSweeps/Err_compEff_SDC_BB_exp_5_lo_eulerS_01_10000_64_5.dat};
	\addlegendentry{sweeps=5}; 
	\addplot[solid,color=olive!85!black,mark=triangle*]
	table[x=stepSize,y=errStr]{data_bb_1_orderSweeps/Err_compEff_SDC_BB_exp_5_lo_eulerS_01_10000_64_6.dat};
	\addlegendentry{sweeps=6}; 
	\addplot+[mark=none,dotted,color=black]{ x^3*0.001 };					\addlegendentry{order $3$};
	\addplot[solid,color=red,mark=square*]
	table[x=stepSize,y=errStr]{data_bb_1_orderSweeps/Err_compEff_SDC_BB_exp_5_lo_eulerS_01_10000_64_7.dat};
	\addlegendentry{sweeps=7}; 
	\addplot[solid,color=orange!85!black,mark=square*]
	table[x=stepSize,y=errStr]{data_bb_1_orderSweeps/Err_compEff_SDC_BB_exp_5_lo_eulerS_01_10000_64_8.dat};
	\addlegendentry{sweeps=8}; 
	\addplot+[mark=none,densely dotted,color=black]{ x^4*0.00025 };		\addlegendentry{order $4$};
\end{axis}
\end{tikzpicture} }
\end{subfigure}
\begin{subfigure}[b]{0.49\textwidth} 
\centering \resizebox{\linewidth}{!}{ 
\centering	\begin{tikzpicture}[scale=0.75]
\begin{axis}[
every axis plot/.style={
line width=1.25pt,
mark size=2.5
},
label style={font=\sffamily},
tick label style={font=\sffamily},
xmode=log,
ymode=log,
log basis x = {2},
xlabel={\large $\Delta T$}, 
ylabel= \large {$ \mathbb{E} \bigl( \bigl\vert \hat{X}_{T}  - X_{T}  \bigr\vert \bigr)  $},
ymin=0.5e-11, ymax = 10,
xtick= data,
legend style={draw=none, at={(0.45,-0.2)},anchor=north, font = \tiny \sffamily, /tikz/every even column/.append style={column sep=0.5cm}}, 
legend columns=3,
transpose legend,
legend to name={legend:convEx1},
domain = 2^-6:2^-3,
]
\addplot[color=red,mark=*]
table[x=stepSize,y=errStr]{data_bb_1_orderSweeps/Err_compEff_SDC_BB_exp_5_lo_eMItoS_01_10000_64_0.dat};
\addlegendentry{initialization};
\addplot[mark=none,loosely dashed]{ 0.35*x^0.5 };   	\addlegendentry{order $0.5$};
\addplot[solid,color=orange,mark=diamond*]
table[x=stepSize,y=errStr]{data_bb_1_orderSweeps/Err_compEff_SDC_BB_exp_5_lo_eMItoS_01_10000_64_1.dat};
\addlegendentry{sweeps=1}; 
\addplot[solid,color=blue,mark=diamond*]
table[x=stepSize,y=errStr]{data_bb_1_orderSweeps/Err_compEff_SDC_BB_exp_5_lo_eMItoS_01_10000_64_2.dat};
\addlegendentry{sweeps=2}; 
\addplot+[mark=none,dashed,color=black]{ 0.055*x };					\addlegendentry{order $1$};
\addplot[solid,color=brown,mark=x]
table[x=stepSize,y=errStr]{data_bb_1_orderSweeps/Err_compEff_SDC_BB_exp_5_lo_eMItoS_01_10000_64_3.dat};
\addlegendentry{sweeps=3}; 
\addplot[solid,color=green!40!black,mark=x]
table[x=stepSize,y=errStr]{data_bb_1_orderSweeps/Err_compEff_SDC_BB_exp_5_lo_eMItoS_01_10000_64_4.dat};
\addlegendentry{sweeps=4}; 
\addplot+[mark=none,dashdotted,color=black]{ 0.007*x^2 };				\addlegendentry{order $2$};
\addplot[solid,color=purple,mark=triangle*]
table[x=stepSize,y=errStr]{data_bb_1_orderSweeps/Err_compEff_SDC_BB_exp_5_lo_eMItoS_01_10000_64_5.dat};
\addlegendentry{sweeps=5}; 
\addplot[solid,color=olive!85!black,mark=triangle*]
table[x=stepSize,y=errStr]{data_bb_1_orderSweeps/Err_compEff_SDC_BB_exp_5_lo_eMItoS_01_10000_64_6.dat};
\addlegendentry{sweeps=6}; 
\addplot+[mark=none,dotted,color=black]{ x^3*0.001  };		\addlegendentry{order $3$};
\addplot[solid,color=red,mark=square*]
table[x=stepSize,y=errStr]{data_bb_1_orderSweeps/Err_compEff_SDC_BB_exp_5_lo_eMItoS_01_10000_64_7.dat};
\addlegendentry{sweeps=7}; 
\addplot[solid,color=orange!85!black,mark=square*]
table[x=stepSize,y=errStr]{data_bb_1_orderSweeps/Err_compEff_SDC_BB_exp_5_lo_eMItoS_01_10000_64_8.dat};
\addlegendentry{sweeps=8}; 
\addplot+[mark=none,densely dotted,color=black]{ x^4*0.00025 };		\addlegendentry{order $4$};
\end{axis}
\end{tikzpicture} }
\end{subfigure}
\ref*{legend:convEx1}
\caption{Accuracy of the numerical approximation of \eqref{ex:TP1} considering the stochastic part only, i.e., $a=0.5$ and $b=1$. We choose SDC-$5$ with a Lobatto time grid, $10^4$ realizations and SDC sweeps from $0$ to $8$ or $9$ depending on the initialization used. The theoretical convergence orders are plotted as dashed lines. 
Left:  initialization using explicit Euler for the SBB-ODE.  Right:  Euler-Maruyama for \Ito-SDE with interpolation initialization. } \label{fig:OrderSweepsAllInits}
\end{figure}

We observe that with every second SDC sweep the convergence order increases by one. This phenomenon can be explained by calculating the method's order considering the ODE 
$\frac{dx}{dt} = cx, \quad x_0 = 1,$
where $c$ is a constant. Let 
$\exp(ct) = 1+ ct + \frac{c^2t^2}{2!} + \frac{c^3t^3}{3!} + ...$
be the exact solution and its Taylor series expansion. When the numerical approximation $\hat{x}^1$ of the ODE is expected to obtain order of accuracy $p=2$, then it must equal the first three terms of the Taylor series expansion of the exact solution. For $c=\frac{\eta_0}{\sqrt{\Delta T}},$ we have the above defined SBB-ODE of \eqref{ex:TP1}. As $\mathbb{E}(\eta_0^k) =0$ if $k$ is odd \cite{papoulis2002probability}, we can approximate the weak local error order after one iteration with time step $\Delta T$ by
\begin{multline}
\lvert\mathbb{E}\left(  x_{\Delta T} - \hat{x}^1_{\Delta T} \right)  \rvert= \\
\bigg\lvert\mathbb{E} \left( \left( 1+ \frac{\eta_0\Delta T}{\sqrt{\Delta T}} + \frac{\eta_0^2\Delta T^2}{2\Delta T} +  \frac{\eta_0^3\Delta T^3}{3!\Delta T^{3/2}} + \frac{\eta_0^4\Delta T^4}{4!\Delta T^2} + \mathcal{O}\left(\Delta T^{5/2}\right) \right) 
- \hat{x}^1(\Delta T) \right)\bigg\rvert \\
= \bigg\lvert \mathbb{E} \left( \frac{\eta_0^3\Delta T^3}{3!\Delta T^{3/2}} + \frac{\eta_0^4\Delta T^4}{4!\Delta T^2} + \mathcal{O}(\Delta T^{5/2}) \right) \bigg\rvert= \frac{3\Delta T^2}{4!} + \mathcal{O}(\Delta T^{5/2}) 
= \mathcal{O} (\Delta T^2).
\end{multline}
Additionally, we can estimate
\begin{multline}
\mathbb{E}\left( \lvert  x_{\Delta T} - \hat{x}^1_{\Delta T} \rvert^2 \right) = \mathbb{E} \left( \left\lvert\frac{\eta_0^3\Delta T^3}{3!\Delta T^{3/2}} + \frac{\eta_0^4\Delta T^4}{4!\Delta T^2} + \mathcal{O}(\Delta T^{5/2}) \right\rvert^2 \right) \\ 
= \mathbb{E} \left( \left\lvert\frac{\eta_0^3\Delta T^3}{3!\Delta T^{3/2}} + \frac{\eta_0^4\Delta T^4}{4!\Delta T^2} \right\rvert^2 +   2 \left(\frac{\eta_0^3\Delta T^3}{3!\Delta T^{3/2}} + \frac{\eta_0^4\Delta T^4}{4!\Delta T^2}  \right) \mathcal{O}(\Delta T^{5/2})  \right) + \mathcal{O}(\Delta T^{5}) \\
= \mathbb{E} \left( \left(\frac{\eta_0^3\Delta T^3}{3!\Delta T^{3/2}} \right)^2 + \left( \frac{\eta_0^4\Delta T^4}{4!\Delta T^2} \right)^2 +  \frac{2\eta_0^7\Delta T^7}{3!\, 4!\Delta T^{7/2}} \right) + \mathcal{O}(\Delta T^{9/2}) \\
= \frac{15}{36} \Delta T^3 + \frac{105}{576} \Delta T^4 + \mathcal{O}(\Delta T^{9/2}) = \frac{5}{12}\Delta T^3 + \mathcal{O}(\Delta T^{4})
\end{multline}
This results in a strong convergence order of $p=3/2-1/2=1$~\cite[Thm.~1.1.1, Rem.~1.1.2]{MilsteinTretyakov2021}, which agrees with the numerical results in Figure~\ref{fig:OrderSweepsAllInits}. With one more sweep, the local weak error has the same order as before, i.e.,
\begin{align*}
\lvert\mathbb{E}\left(  x_{\Delta T} - \hat{x}^2_{\Delta T} \right) \rvert &= \bigg\lvert\mathbb{E} \left( \frac{\eta_0^4\Delta T^4}{4!\Delta T^2} + \mathcal{O}(\Delta T^{5/2}) \right) \bigg\rvert= \frac{3\Delta T^2}{4!} + \mathcal{O}(\Delta T^{5/2}) 
= \mathcal{O} (\Delta T^2),
\end{align*}
and also the strong order does not increase~\cite[Thm.~1.1.1]{MilsteinTretyakov2021}.
We can observe two effects influencing the convergence order. Due to the denominator $\sqrt{\Delta T}$ of $B^{0}$ the order of accuracy is halved compared to the order of the method applied to an ODE, as predicted by the convergence analysis. Since $\mathbb{E}(\eta_0^k) =0$, if $k$ is odd, the theoretical order of $\tfrac{1}{2}$ is not visible. These two effects cause the order to increase only every second iteration by one. The initialization will obtain the order of the underlying method, i.e., zero for the constant (not shown) and Euler SBB-ODE initialization, and $\tfrac{1}{2}$ for Euler-Maruyama for \Ito-SDE with interpolation.

Figure \ref{fig:ConvOrder_TP1} shows the order of accuracy for varying number of collocation nodes. Additionally, we compare the Milstein and Euler-Maruyama scheme with the SDC methods with $3,4,5$ Lobatto nodes, using the Euler-Maruyama with interpolation initialization. 
\begin{figure}[htb!]
\centering
\begin{subfigure}[b]{0.49\textwidth} 
\centering \resizebox{\linewidth}{!}{ 
\begin{tikzpicture}
\begin{axis}[
every axis plot/.style={
line width=1.25pt,
mark size=2.5,
},
xmode=log,
ymode=log,
log basis x = {2},
xlabel=$\Delta T$, 
ylabel= \large {$ \mathbb{E} \bigl( \bigl\vert \hat{X}_{T}  - X_{T}  \bigr\vert \bigr)  $},
ymin=1e-12, ymax=10^0,
xmin=2^-11, xmax=2^-0,
try min ticks,
max space between ticks=35,
legend style={draw=none, at={(0.45,-0.2)},anchor=north, font = \small \sffamily, /tikz/every even column/.append style={column sep=0.5cm}}, 
legend columns=2,
transpose legend,
legend to name={legend:WeakConvOrder1},
domain = 3.90625e-03:6.25e-02,
]
\addplot[color=red,mark=*]
table[x=stepSize,y=errStr]{data_bb_1/Err_compEff_EM_exp_3_eq_constI_01_100000_1024.dat};
\addlegendentry{Euler-Maruyama (order 0.5)}; 
\addplot[solid,color=orange,mark=triangle*]
table[x=stepSize,y=errStr]{data_bb_1/Err_compEff_Mil_exp_3_eq_constI_01_100000_1024.dat};
\addplot[solid,color=blue,mark=diamond*]
table[x=stepSize,y=errStr]{data_bb_1/Err_compEff_SDC_BB_exp_3_lo_eMItoS_01_10000_1024.dat};
\addlegendentry{SDC-3 (order 2)}; 
\addplot[solid,color=brown,mark=square*]
table[x=stepSize,y=errStr]{data_bb_1/Err_compEff_SDC_BB_exp_4_lo_eMItoS_01_10000_1024.dat};
\addlegendentry{SDC-4 (order 3)}; 
\addplot[solid,color=green!40!black,mark=x]
table[x=stepSize,y=errStr]{data_bb_1/Err_compEff_SDC_BB_exp_5_lo_eMItoS_01_10000_1024.dat};
\addlegendentry{SDC-5 (order 4)}; 
\addplot+[mark=none,densely dashed,color=black, domain=2^-6:2^-1]{ x^4*0.00007 };	
\addplot+[mark=none, dash dot dot,color=black, domain=2^-9:2^-1]{ 0.001*x^3};	
\addplot+[mark=none,dash dot,color=black, domain=2^-10:2^-1]{ 0.025*x^2};	
\addplot+[mark=none,dashed,color=black, domain=2^-10:2^-1]{ 0.22*x };		
\addplot+[mark=none,loosely dotted,color=black, domain=2^-10:2^-1]{ 0.3*x^0.5 };		
\end{axis}
\end{tikzpicture} } 
\end{subfigure} 
\begin{subfigure}[b]{0.49\textwidth} \centering \resizebox{\linewidth}{!}{ 
\begin{tikzpicture}
\begin{axis}[
every axis plot/.style={
line width=1.25pt,
mark size=2.5,
},
xmode=log,
ymode=log,
log basis x = {2},
xlabel=$\Delta T$, 
ylabel= \large{$\bigl\vert \mathbb{E} \bigl( 
\hat{X}_{T}\bigr)  - \mathbb{E} \bigl(X_{T}  \bigr) \bigr\vert $},
ymin=1e-12, ymax=10^0,
xmin=2^-11, xmax=2^-0,
try min ticks,
max space between ticks=35,
legend style={draw=none, at={(0.45,-0.2)},anchor=north, font = \small \sffamily, /tikz/every even column/.append style={column sep=0.5cm}}, 
legend columns=2,
transpose legend,
legend to name={legend:WeakConvOrder},
domain = 3.90625e-03:6.25e-02,
]
\addplot[color=red,mark=*]
table[x=stepSize,y=errT]{data_bb_1/Err_compEff_EM_exp_3_eq_constI_01_100000_1024.dat};
\addlegendentry{Euler-Maruyama (order 0.5 / 1)}; 
\addplot[solid,color=orange,mark=triangle*]
table[x=stepSize,y=errT]{data_bb_1/Err_compEff_Mil_exp_3_eq_constI_01_100000_1024.dat};
\addlegendentry{Milstein (order 1)}; 
\addplot[solid,color=blue,mark=diamond*]
table[x=stepSize,y=errT]{data_bb_1/Err_compEff_SDC_BB_exp_3_lo_eMItoS_01_10000_1024.dat};
\addlegendentry{SDC-3 (order 2)}; 
\addplot[solid,color=brown,mark=square*]
table[x=stepSize,y=errT]{data_bb_1/Err_compEff_SDC_BB_exp_4_lo_eMItoS_01_10000_1024.dat};
\addlegendentry{SDC-4 (order 3)}; 
\addplot[solid,color=green!40!black,mark=x]
table[x=stepSize,y=errT]{data_bb_1/Err_compEff_SDC_BB_exp_5_lo_eMItoS_01_10000_1024.dat};
\addlegendentry{SDC-5 (order 4)}; 
\addplot+[mark=none,densely dashed,color=black, domain=2^-6:2^-1]{ x^4*0.000055 };	
\addplot+[mark=none, dash dot dot,color=black, domain=2^-9:2^-1]{ 0.0002*x^3};	
\addplot+[mark=none,dash dot,color=black, domain=2^-10:2^-1]{ 0.02*x^2};	
\addplot+[mark=none,dashed,color=black, domain=2^-10:2^-1]{ 0.09*x };		
\end{axis}
\end{tikzpicture} } 
\end{subfigure} 
\ref*{legend:WeakConvOrder}
\caption{Strong  (left) and weak (right) convergence for the linear scalar SDE example. 
We used $10^4$ realizations for the SDC-$p$ methods and $10^5$ for Euler-Maruyama and Milstein scheme to determine the expected error.} \label{fig:ConvOrder_TP1} 
\end{figure}
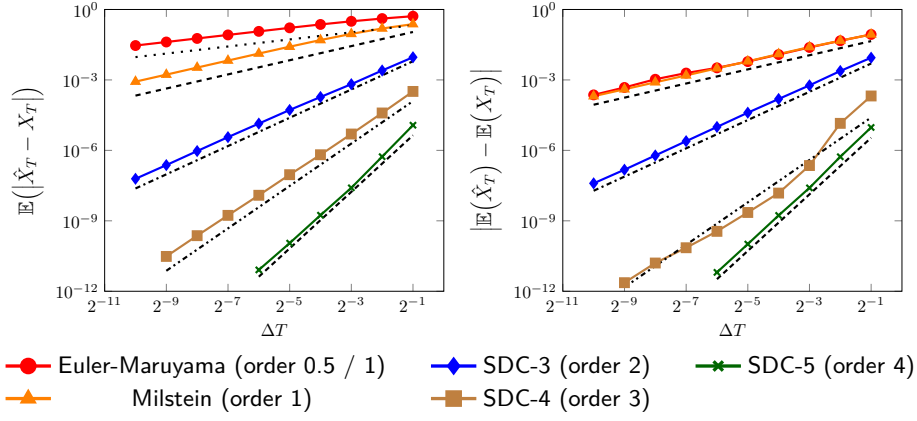
Computational efficiency of the method is illustrated in Figure~\ref{fig:CompEff_TP1}. 
For all schemes, the effort considering strong convergence  is somewhat higher than the effort considering weak convergence.  Aiming at a moderate, fixed  accuracy, we observe higher efficiency using a higher-order SDC method, i.e., more collocation nodes and sweeps. Compared to the Euler-Maruyama and Milstein scheme, the SDC method achieves higher accuracy for a fixed number of function evaluations, except for rather large error tolerances.

\begin{figure}[htb!]
\centering
\begin{subfigure}[b]{0.48\textwidth} 
\centering \resizebox{\linewidth}{!}{ 
\begin{tikzpicture}
\begin{axis}[
every axis plot/.style={
line width=1.25pt,
mark size=2.5,
},    
xmode=log,
ymode=log,
legend style={draw=none, at={(0.45,-0.2)},anchor=north, font =\small \sffamily, /tikz/every even column/.append style={column sep=0.5cm}}, 
legend columns=2,
transpose legend,
xlabel={function evaluations},    
ylabel= \large {$ \mathbb{E} \bigl( \bigl\vert \hat{X}_{T}  - X_{T}  \bigr\vert \bigr)  $},   
ymin=1e-12, ymax=10^0,
xmin=10^1,
xmax=10^5,
try min ticks,
max space between ticks=50
]
\addplot[color=red,mark=*] table[x=fctEval,y=errStr]{data_bb_1/compEff/Err_compEff_EM_exp_3_eq_constI_01_100000_1024.dat};
\addplot[color=orange,mark=triangle*] table[x=fctEval,y=errStr]{data_bb_1/compEff/Err_compEff_Mil_exp_3_eq_constI_01_100000_1024.dat};
\addplot[color=blue,mark=diamond*] table[x=fctEval,y=errStr]{data_bb_1/compEff/Err_compEff_SDC_BB_exp_3_lo_eMItoS_01_10000_1024.dat};
\addplot[color=brown,mark=square*] table[x=fctEval,y=errStr]{data_bb_1/compEff/Err_compEff_SDC_BB_exp_4_lo_eMItoS_01_10000_1024.dat};
\addplot[color=green!40!black,mark=x] table[x=fctEval,y=errStr]{data_bb_1/compEff/Err_compEff_SDC_BB_exp_5_lo_eMItoS_01_10000_1024.dat};
\end{axis}
\end{tikzpicture}}
\end{subfigure}
\begin{subfigure}[b]{0.48\textwidth} 
\centering \resizebox{\linewidth}{!}{ 
\begin{tikzpicture}
\begin{axis}[
every axis plot/.style={
line width=1.25pt,
mark size=2.5,
},    
xmode=log,
ymode=log,
ymin=1e-12, ymax=10^0,
xmin=10^1,
xmax=10^5,
legend style={draw=none, at={(0.45,-0.2)},anchor=north, font =\small  \sffamily, /tikz/every even column/.append style={column sep=0.5cm}}, 
legend columns=2,
transpose legend,
legend to name={legend:CompEffTP1},
xlabel={function evaluations},    
ylabel= \large{$\bigl\vert \mathbb{E} \bigl( \hat{X}_{T}\bigr)  - \mathbb{E} \bigl(X_{T}  \bigr) \bigr\vert $},   
try min ticks,
max space between ticks=50
]
\addplot[color=red,mark=*] table[x=fctEval,y=errT]{data_bb_1/compEff/Err_compEff_EM_exp_3_eq_constI_01_100000_1024.dat};
\addlegendentry{Euler-Maruyama}; 
\addplot[color=orange,mark=triangle*] table[x=fctEval,y=errT]{data_bb_1/compEff/Err_compEff_Mil_exp_3_eq_constI_01_100000_1024.dat};
\addlegendentry{Milstein}; 
\addplot[color=blue,mark=diamond*] table[x=fctEval,y=errT]{data_bb_1/compEff/Err_compEff_SDC_BB_exp_3_lo_eMItoS_01_10000_1024.dat};
\addlegendentry{SDC-3}; 
\addplot[color=brown,mark=square*] table[x=fctEval,y=errT]{data_bb_1/compEff/Err_compEff_SDC_BB_exp_4_lo_eMItoS_01_10000_1024.dat};
\addlegendentry{SDC-4}; 
\addplot[color=green!40!black,mark=x] table[x=fctEval,y=errT]{data_bb_1/compEff/Err_compEff_SDC_BB_exp_5_lo_eMItoS_01_10000_1024.dat};
\addlegendentry{SDC-5}; 
\end{axis}
\end{tikzpicture}}
\end{subfigure}
\ref*{legend:CompEffTP1}
\caption{Computational effort shown by function evaluations per realization, and strong (left ) and weak convergence (right) for the linear scalar SDE example using 
Euler-Maruyama initialization for the \Ito-SDE with interpolation. We used $10^4$ realizations for the SDC-$p$ methods and $10^5$ for Euler-Maruyama and Milstein scheme to determine the expected error. } \label{fig:CompEff_TP1}
\end{figure}
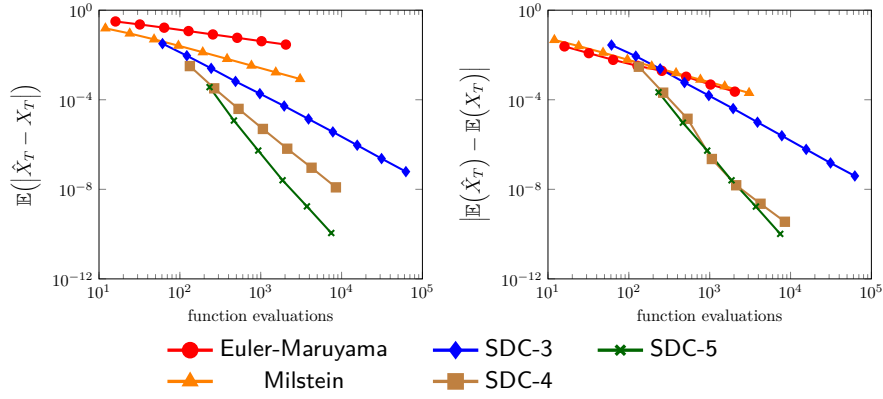

\subsubsection{Nonlinear scalar SDE.}
Consider the nonlinear scalar SDE~\cite{tianTwoStageStochasticRungeKutta2002}
\begin{equation}
dX_t = -a\left(1-X_t^2\right)dt + b\left(1-X_t^2\right)\circ dW_t,
\label{ex:TP2}
\end{equation}
with  initial value $X_0 = 0$ and exact solution
\begin{equation}
X_t =\frac{\left( X_0 +1 \right)\exp\left( -2at+2bW_t \right)+X_0-1}{\left( X_0 +1 \right)\exp\left( -2at+2bW_t \right) -X_0+1}.
\end{equation}
We choose $a=b=1$, and use $B^{(0)}(t)$ as before. Comparing to implicit Euler-Maruyama and Milstein schemes, we observe similar behavior as for the linear SDE case. 
Figure \ref{fig:ConvOrder_TP2} shows the order of accuracy of the implicit Euler-Maruyama, implicit Milstein and SDC-$p$ method, $p=3,4,5$. 
\begin{figure}[tbh!]
\centering
\includegraphics[scale=0.6]{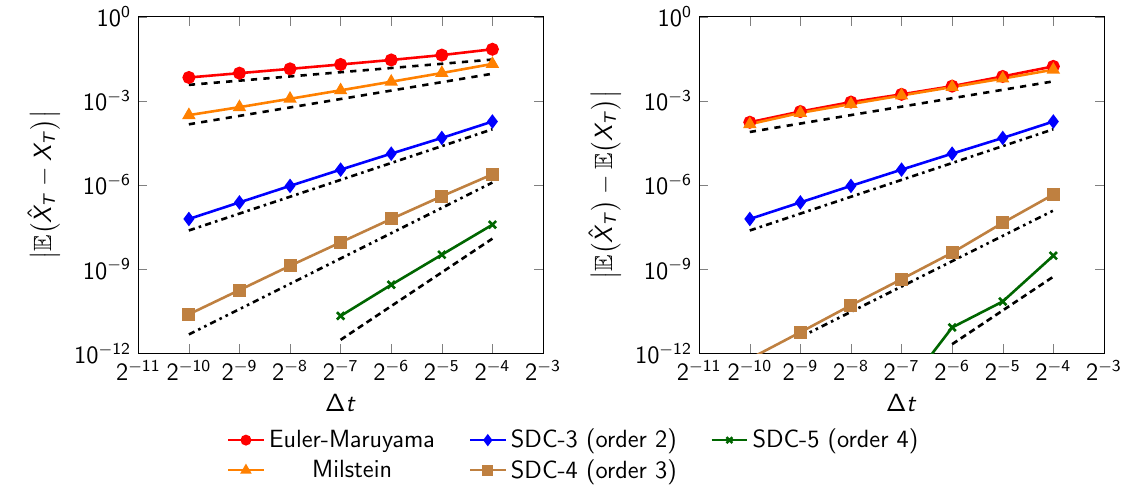}
\caption{Strong  (left) and weak convergence  (right) for the nonlinear SDE example using the Euler for SBB-ODE initialization. 
We used $10^5$ realizations for the SDC-$p$ methods, implicit Euler-Maruyama and implicit Milstein scheme to determine the expected error.
} \label{fig:ConvOrder_TP2} 
\end{figure}

\begin{figure}[tbh!]
\centering
\includegraphics[scale=0.6]{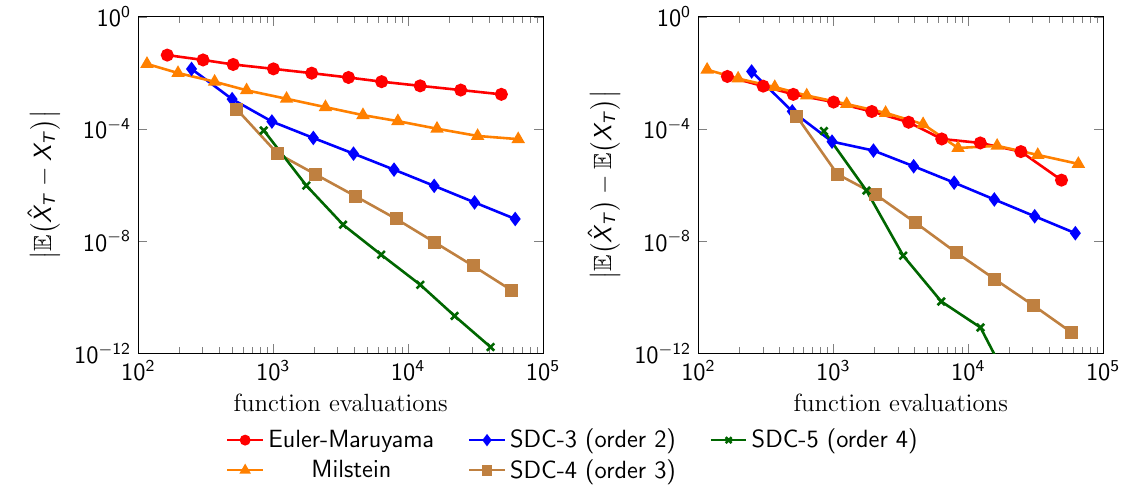}
\caption{Computational effort (function evaluations per realization) for strong (left) and weak (right) error  for the nonlinear scalar SDE example using an Euler for SBB-ODE intialization. 
We used $10^5$ realizations for the SDC-$p$ methods, Euler-Maruyama and Milstein scheme to determine the expected error.} \label{fig:CompEff_TP2}
\end{figure}
Computational effort for the different methods is shown in Figure \ref{fig:CompEff_TP2}.
Except for high error tolerances, SDC requires fewer function evaluations compared to Euler-Maruyama and Milstein schemes. 

\subsubsection{Two-dimensional linear SDE with commutative noise.} We investigate 
the linear two-dimensional SDE
\begin{equation}
dX_t = \left( A-\frac{1}{2} \left( B_1^2 + B_2^2 \right) \right) X_t dt + B_1 X_t \circ dW^1_t+ B_2 X_t \circ dW^2_t,
\label{ex:TP3}
\end{equation}
with exact solution
\begin{align*}
X_t &=\exp \left( \left( A-\frac{1}{2} \left( B_1^2 + B_2^2 \right) \right) t + B_1  W^1_t+ B_2 W^2_t \right) X_0,
\end{align*}
where $A,B_1, B_2$ are commutable $2 \times 2$ matrices.
The initial value and coefficient matrices are:
\begin{align*}
X_0 = \begin{pmatrix} 1 \\ 0 \end{pmatrix}, \quad A = \begin{pmatrix} 1.5 & -0.85  \\ 1.275 & -0.625 \end{pmatrix}, \quad B_1 = \begin{pmatrix} 0.9 & -0.2  \\ 0.3 & 0.4 \end{pmatrix}, \quad B_2 = \begin{pmatrix} 2.1 & -1.2  \\ 1.8 & -0.9 \end{pmatrix}.
\end{align*}

Strong and weak convergence for each component of the SDE system is illustrated in Figure \ref{fig:ConvOrder_TP4}, with similar behavior as in the scalar case.
\begin{figure}[tbh!]
\centering
\begin{subfigure}[b]{0.48\textwidth} \centering \resizebox{\linewidth}{!}{ 
\begin{tikzpicture}
\begin{axis}[
every axis plot/.style={
line width=1.25pt,
mark size=2.5,
},
xmode=log,
ymode=log,
log basis x = {2},
xlabel=$\Delta T$, 
ylabel=  \large{$ \mathbb{E} \bigl( \bigl\vert
\hat{X}^1_{T}  - X^1_{T} \bigr\vert  \bigr) $},
ymin=1e-12, ymax=10^0,
xmin=2^-11, xmax=2^-3,
try min ticks,
max space between ticks=35,
legend style={draw=none, at={(0.45,-0.2)},anchor=north, font = \small \sffamily, /tikz/every even column/.append style={column sep=0.5cm}}, 
legend columns=2,
transpose legend,
legend to name={legend:ConvOrder_TP4_X1},
]
\addplot[color=red,mark=*]
table[x=stepSize,y=errStr]{data_bb_1/Err_compEff_X1_EM_TP4_3_eq_constI_01_100000_16384_1_4.dat};
\addlegendentry{Euler-Maruyama}; 
\addplot[solid,color=orange,mark=triangle*]
table[x=stepSize,y=errStr]{data_bb_1/Err_compEff_X1_Mil_TP4_3_eq_constI_01_100000_16384_1_4.dat};
\addlegendentry{Milstein}; 
\addplot[solid,color=blue,mark=diamond*]
table[x=stepSize,y=errStr]{data_bb_1/Err_compEff_X1_SDC_BB_TP4_3_lo_eulerS_01_10000_16384.dat};
\addlegendentry{SDC-3 (order 2)}; 
\addplot[solid,color=brown,mark=square*]
table[x=stepSize,y=errStr]{data_bb_1/Err_compEff_X1_SDC_BB_TP4_4_lo_eulerS_01_10000_16384.dat};
\addlegendentry{SDC-4 (order 3)}; 
\addplot[solid,color=green!40!black,mark=x]
table[x=stepSize,y=errStr]{data_bb_1/Err_compEff_X1_SDC_BB_TP4_5_lo_eulerS_01_10000_16384.dat};
\addlegendentry{SDC-5 (order 4)}; 
\addplot+[mark=none,densely dashed,color=black, domain=2^-7:2^-4]{ x^4*0.001 };	
\addplot+[mark=none, dash dot dot,color=black, domain=2^-10:2^-4]{ 0.01*x^3};	
\addplot+[mark=none,dash dot,color=black, domain=2^-10:2^-4]{ 0.25*x^2};	
\addplot+[mark=none,dashed,color=black,, domain=2^-10:2^-4]{ 1.3*x };		
\addplot+[mark=none,loosely dotted,color=black, domain=2^-10:2^-4]{ 1.3*x^0.5 };		
\end{axis}
\end{tikzpicture} } 
\end{subfigure}  
\begin{subfigure}[b]{0.48\textwidth} 
\centering \resizebox{\linewidth}{!}{ 
\begin{tikzpicture}
\begin{axis}[
every axis plot/.style={
line width=1.25pt,
mark size=2.5,
},
xmode=log,
ymode=log,
log basis x = {2},
xlabel=$\Delta T$, 
ylabel= \large{ $ \bigl\vert \mathbb{E} \bigl( 
\hat{X}^1_{T} \bigr) - \mathbb{E} \bigl( X^1_{T}  \bigr)  \bigr\vert$},
ymin=1e-12, ymax=10^0,
xmin=2^-11, xmax=2^-3,
try min ticks,
max space between ticks=35,
legend style={draw=none, at={(0.45,-0.2)},anchor=north, font = \tiny \sffamily, /tikz/every even column/.append style={column sep=0.5cm}}, 
legend columns=2,
transpose legend,
]
\addplot[color=red,mark=*]
table[x=stepSize,y=errT]{data_bb_1/Err_compEff_X1_EM_TP4_3_eq_constI_01_100000_16384_1_4.dat};
\addplot[solid,color=orange,mark=triangle*]
table[x=stepSize,y=errT]{data_bb_1/Err_compEff_X1_Mil_TP4_3_eq_constI_01_100000_16384_1_4.dat};
\addplot[solid,color=blue,mark=diamond*]
table[x=stepSize,y=errT]{data_bb_1/Err_compEff_X1_SDC_BB_TP4_3_lo_eulerS_01_10000_16384.dat};
\addplot[solid,color=brown,mark=square*]
table[x=stepSize,y=errT]{data_bb_1/Err_compEff_X1_SDC_BB_TP4_4_lo_eulerS_01_10000_16384.dat};
\addplot[solid,color=green!40!black,mark=x]
table[x=stepSize,y=errT]{data_bb_1/Err_compEff_X1_SDC_BB_TP4_5_lo_eulerS_01_10000_16384.dat};
\addplot+[mark=none, dash dot dot,color=black, domain=2^-10:2^-4]{ 0.002*x^3};	
\addplot+[mark=none,dash dot,color=black, domain=2^-10:2^-4]{ 0.25*x^2};	
\addplot+[mark=none,dashed,color=black, domain=2^-10:2^-4]{ 0.5*x };		
\addplot+[mark=none,densely dashed,color=black, domain=2^-7:2^-4]{ x^4*0.0008 };	
\end{axis}
\end{tikzpicture} } 
\end{subfigure} 
\begin{subfigure}[b]{0.48\textwidth} \centering \resizebox{\linewidth}{!}{ 
\begin{tikzpicture}
\begin{axis}[
every axis plot/.style={
line width=1.25pt,
mark size=2.5,
},
xmode=log,
ymode=log,
log basis x = {2},
xlabel=$\Delta T$, 
ylabel= \large{ $ \mathbb{E} \bigl( \bigl\vert
\hat{X}^2_{T}  - X^2_{T} \bigr\vert  \bigr) $},
ymin=1e-12, ymax=10^0,
xmin=2^-11, xmax=2^-3,
try min ticks,
max space between ticks=35,
legend style={draw=none, at={(0.45,-0.2)},anchor=north, font = \tiny \sffamily, /tikz/every even column/.append style={column sep=0.5cm}}, 
]
\addplot[color=red,mark=*]
table[x=stepSize,y=errStr]{data_bb_1/Err_compEff_X2_EM_TP4_3_eq_constI_01_100000_16384_1_4.dat};
\addplot[solid,color=orange,mark=triangle*]
table[x=stepSize,y=errStr]{data_bb_1/Err_compEff_X2_Mil_TP4_3_eq_constI_01_100000_16384_1_4.dat};
\addplot[solid,color=blue,mark=diamond*]
table[x=stepSize,y=errStr]{data_bb_1/Err_compEff_X2_SDC_BB_TP4_3_lo_eulerS_01_10000_16384.dat};
\addplot[solid,color=brown,mark=square*]
table[x=stepSize,y=errStr]{data_bb_1/Err_compEff_X2_SDC_BB_TP4_4_lo_eulerS_01_10000_16384.dat};
\addplot[solid,color=green!40!black,mark=x]
table[x=stepSize,y=errStr]{data_bb_1/Err_compEff_X2_SDC_BB_TP4_5_lo_eulerS_01_10000_16384.dat};
\addplot+[mark=none,densely dashed,color=black, domain=2^-7:2^-4]{ x^4*0.001 };	
\addplot+[mark=none, dash dot dot,color=black, domain=2^-10:2^-4]{ 0.01*x^3};	
\addplot+[mark=none,dash dot,color=black, domain=2^-10:2^-4]{ 0.25*x^2};	
\addplot+[mark=none,dashed,color=black,, domain=2^-10:2^-4]{ 1.3*x };		
\addplot+[mark=none,loosely dotted,color=black, domain=2^-10:2^-4]{ 1.3*x^0.5 };		
\end{axis}
\end{tikzpicture} } 
\end{subfigure}
\begin{subfigure}[b]{0.48\textwidth} 
\centering \resizebox{\linewidth}{!}{ 
\begin{tikzpicture}
\begin{axis}[
every axis plot/.style={
line width=1.25pt,
mark size=2.5,
},
xmode=log,
ymode=log,
log basis x = {2},
xlabel=$\Delta T$, 
ylabel= \large{ $ \bigl\vert \mathbb{E} \bigl( 
\hat{X}^2_{T} \bigr) - \mathbb{E} \bigl( X^2_{T}  \bigr)  \bigr\vert$},
ymin=1e-12, ymax=10^0,
xmin=2^-11, xmax=2^-3,
try min ticks,
max space between ticks=35,
legend style={draw=none, at={(0.45,-0.2)},anchor=north, font = \small \sffamily, /tikz/every even column/.append style={column sep=0.5cm}}, 
legend columns=2,
transpose legend,
domain = 3.90625e-03:6.25e-02,
]
\addplot[color=red,mark=*]
table[x=stepSize,y=errT]{data_bb_1/Err_compEff_X2_EM_TP4_3_eq_constI_01_100000_16384_1_4.dat};
\addplot[solid,color=orange,mark=triangle*]
table[x=stepSize,y=errT]{data_bb_1/Err_compEff_X2_Mil_TP4_3_eq_constI_01_100000_16384_1_4.dat};
\addplot[solid,color=blue,mark=diamond*]
table[x=stepSize,y=errT]{data_bb_1/Err_compEff_X2_SDC_BB_TP4_3_lo_eulerS_01_10000_16384.dat};
\addplot[solid,color=brown,mark=square*]
table[x=stepSize,y=errT]{data_bb_1/Err_compEff_X2_SDC_BB_TP4_4_lo_eulerS_01_10000_16384.dat};
\addplot[solid,color=green!40!black,mark=x]
table[x=stepSize,y=errT]{data_bb_1/Err_compEff_X2_SDC_BB_TP4_5_lo_eulerS_01_10000_16384.dat};
\addplot+[mark=none, dash dot dot,color=black, domain=2^-10:2^-4]{ 0.002*x^3};	
\addplot+[mark=none,dash dot,color=black, domain=2^-10:2^-4]{ 0.25*x^2};	
\addplot+[mark=none,dashed,color=black, domain=2^-10:2^-4]{ 0.5*x };		
\addplot+[mark=none,densely dashed,color=black, domain=2^-7:2^-4]{ x^4*0.0008 };	
\end{axis}
\end{tikzpicture} } 
\end{subfigure} 
\ref*{legend:ConvOrder_TP4_X1}
\caption{Strong (left column) and weak convergence order (right column) for the two dimensional example for component $X^1$ (first row) and $X^2$ (second row) using the Euler for SBB-ODE initialization. The theoretical convergence orders are plotted as dashed lines. We used $10^4$ realizations for the SDC-$p$ methods and $10^5$ for Euler-Maruyama and Milstein scheme to determine the expected error.} \label{fig:ConvOrder_TP4} 
\end{figure}
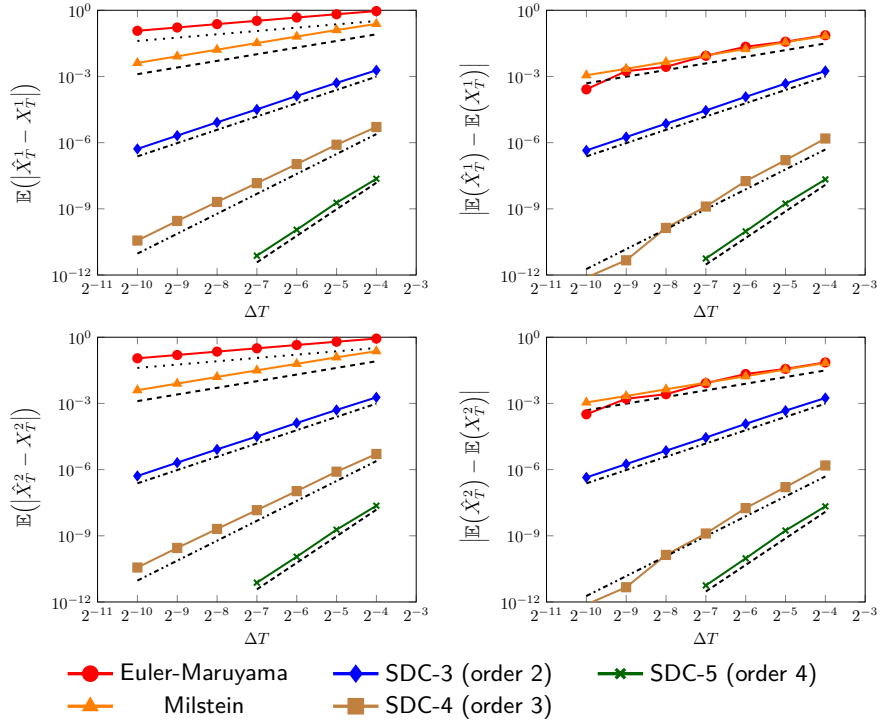 
Computational effort for each component $X^1$ and $X^2$ is shown in Figure \ref{fig:CompEff_TP4}. 
For a strong error tolerance of $10^{-4}$, SDC$-3$ reduces the function evaluations by a factor around 10 compared to the Milstein scheme. 
\begin{figure}[tbh!]
\centering
\begin{subfigure}[b]{0.48\textwidth} 
\centering \resizebox{\linewidth}{!}{ 
\begin{tikzpicture}
\begin{axis}[
every axis plot/.style={
line width=1.25pt,
mark size=2.5,
},    
xmode=log,
ymode=log,
legend style={draw=none, at={(0.45,-0.2)},anchor=north, font =\small , /tikz/every even column/.append style={column sep=0.5cm}}, 
legend columns=2,
transpose legend,
legend to name ={legend:CompEff_TP4_X1},
xlabel={function evaluations},    
ylabel= \large{$\bigl\vert \mathbb{E} \bigl( \hat{X}^1_{T}  - X^1_{T}  \bigr) \bigr\vert $ },   
ymin=1e-12, ymax=10^0,
xmin=5*10^2,
xmax=10^5,
try min ticks,
max space between ticks=50
]
\addplot[color=red,mark=*] table[x=fctEval,y=errStr]{data_bb_1/compEff/Err_compEff_X1_EM_TP4_3_eq_constI_01_100000_16384_1_4.dat};
\addlegendentry{Euler-Maruyama }; 
\addplot[color=orange,mark=triangle*] table[x=fctEval,y=errStr]{data_bb_1/compEff/Err_compEff_X1_Mil_TP4_3_eq_constI_01_100000_16384_1_4.dat};
\addlegendentry{Milstein}; 
\addplot[color=blue,mark=diamond*] table[x=fctEval,y=errStr]{data_bb_1/compEff/Err_compEff_X1_SDC_BB_TP4_3_lo_eulerS_01_10000_16384.dat};
\addlegendentry{SDC-3 (order 2)}; 
\addplot[color=brown,mark=square*] table[x=fctEval,y=errStr]{data_bb_1/compEff/Err_compEff_X1_SDC_BB_TP4_4_lo_eulerS_01_10000_16384.dat};
\addlegendentry{SDC-4 (order 3)};
\addplot[color=green!40!black,mark=x] table[x=fctEval,y=errStr]{data_bb_1/compEff/Err_compEff_X1_SDC_BB_TP4_5_lo_eulerS_01_10000_16384.dat};
\addlegendentry{SDC-5 (order 4)};
\end{axis}
\end{tikzpicture}}
\end{subfigure}
\begin{subfigure}[b]{0.48\textwidth} 
\centering \resizebox{\linewidth}{!}{ 
\begin{tikzpicture}
\begin{axis}[
every axis plot/.style={
line width=1.25pt,
mark size=2.5,
},    
xmode=log,
ymode=log,
ymax=10^0,
ymin=1e-10,
xmin=5*10^2,
xmax=10^5,
legend style={draw=none, at={(0.45,-0.2)},anchor=north, font =\tiny  \sffamily, /tikz/every even column/.append style={column sep=0.5cm}}, 
legend columns=2,
transpose legend,
xlabel={function evaluations},    
ylabel= \large{$\bigl\vert \mathbb{E} \bigl( \hat{X}^1_{T}\bigr)  - \mathbb{E} \bigl(X^1_{T}  \bigr) \bigr\vert $},   
ymin=1e-12, ymax=10^0,
try min ticks,
max space between ticks=50
]
\addplot[color=red,mark=*] table[x=fctEval,y=errT]{data_bb_1/compEff/Err_compEff_X1_EM_TP4_3_eq_constI_01_100000_16384_1_4.dat};
\addplot[color=orange,mark=triangle*] table[x=fctEval,y=errT]{data_bb_1/compEff/Err_compEff_X1_Mil_TP4_3_eq_constI_01_100000_16384_1_4.dat};
\addplot[color=blue,mark=diamond*] table[x=fctEval,y=errT]{data_bb_1/compEff/Err_compEff_X1_SDC_BB_TP4_3_lo_eulerS_01_10000_16384.dat};
\addplot[color=brown,mark=square*] table[x=fctEval,y=errT]{data_bb_1/compEff/Err_compEff_X1_SDC_BB_TP4_4_lo_eulerS_01_10000_16384.dat};
\addplot[color=green!40!black,mark=x] table[x=fctEval,y=errT]{data_bb_1/compEff/Err_compEff_X1_SDC_BB_TP4_5_lo_eulerS_01_10000_16384.dat};
\end{axis}
\end{tikzpicture}}
\end{subfigure}
\begin{subfigure}[b]{0.48\textwidth} 
\centering \resizebox{\linewidth}{!}{ 
\begin{tikzpicture}
\begin{axis}[
every axis plot/.style={
line width=1.25pt,
mark size=2.5,
},    
xmode=log,
ymode=log,
legend style={draw=none, at={(0.45,-0.2)},anchor=north, font =\large \sffamily, /tikz/every even column/.append style={column sep=0.5cm}}, 
legend columns=2,
transpose legend,
xlabel={function evaluations},    
ylabel=\large{$ \mathbb{E} \bigl( \bigl\vert \hat{X}^2_{T}  - X^2_{T}  \bigr\vert  \bigr)$ },   
ymin=1e-12, ymax=10^0,
xmin=5*10^2,
xmax=10^5,
try min ticks,
max space between ticks=50
]
\addplot[color=red,mark=*] table[x=fctEval,y=errStr]{data_bb_1/compEff/Err_compEff_X2_EM_TP4_3_eq_constI_01_100000_16384_1_4.dat};
\addplot[color=orange,mark=triangle*] table[x=fctEval,y=errStr]{data_bb_1/compEff/Err_compEff_X2_Mil_TP4_3_eq_constI_01_100000_16384_1_4.dat};
\addplot[color=blue,mark=diamond*] table[x=fctEval,y=errStr]{data_bb_1/compEff/Err_compEff_X2_SDC_BB_TP4_3_lo_eulerS_01_10000_16384.dat};
\addplot[color=brown,mark=square*] table[x=fctEval,y=errStr]{data_bb_1/compEff/Err_compEff_X2_SDC_BB_TP4_4_lo_eulerS_01_10000_16384.dat};
\addplot[color=green!40!black,mark=x] table[x=fctEval,y=errStr]{data_bb_1/compEff/Err_compEff_X2_SDC_BB_TP4_5_lo_eulerS_01_10000_16384.dat};
\end{axis}
\end{tikzpicture}}
\end{subfigure}
\begin{subfigure}[b]{0.48\textwidth} 
\centering \resizebox{\linewidth}{!}{ 
\begin{tikzpicture}
\begin{axis}[
every axis plot/.style={
line width=1.25pt,
mark size=2.5,
},    
xmode=log,
ymode=log,
ymin=1e-12, ymax=10^0,
xmin=5*10^2,
xmax=10^5,
legend style={draw=none, at={(0.45,-0.2)},anchor=north, font =\tiny  \sffamily, /tikz/every even column/.append style={column sep=0.5cm}}, 
legend columns=2,
transpose legend,
xlabel={function evaluations},    
ylabel=  \large{$\bigl\vert \mathbb{E} \bigl( \hat{X}^1_{T}\bigr)  - \mathbb{E} \bigl(X^1_{T}  \bigr) \bigr\vert $},   
try min ticks,
max space between ticks=50
]
\addplot[color=red,mark=*] table[x=fctEval,y=errT]{data_bb_1/compEff/Err_compEff_X2_EM_TP4_3_eq_constI_01_100000_16384_1_4.dat};
\addplot[color=orange,mark=triangle*] table[x=fctEval,y=errT]{data_bb_1/compEff/Err_compEff_X2_Mil_TP4_3_eq_constI_01_100000_16384_1_4.dat};
\addplot[color=blue,mark=diamond*] table[x=fctEval,y=errT]{data_bb_1/compEff/Err_compEff_X2_SDC_BB_TP4_3_lo_eulerS_01_10000_16384.dat};
\addplot[color=brown,mark=square*] table[x=fctEval,y=errT]{data_bb_1/compEff/Err_compEff_X2_SDC_BB_TP4_4_lo_eulerS_01_10000_16384.dat};
\addplot[color=green!40!black,mark=x] table[x=fctEval,y=errT]{data_bb_1/compEff/Err_compEff_X2_SDC_BB_TP4_5_lo_eulerS_01_10000_16384.dat};
\end{axis}
\end{tikzpicture}}
\end{subfigure}
\ref*{legend:CompEff_TP4_X1}
\caption{Computational effort (function evaluations per realization) for the 2D example,  showing  strong (left column) and weak (right column)  error for components $X^1$ (first row) and $X^2$ (second row). 
We used $10^4$ realizations for the SDC-$p$ methods and $10^5$ for Euler-Maruyama and Milstein scheme to determine the expected error.} \label{fig:CompEff_TP4}
\end{figure}
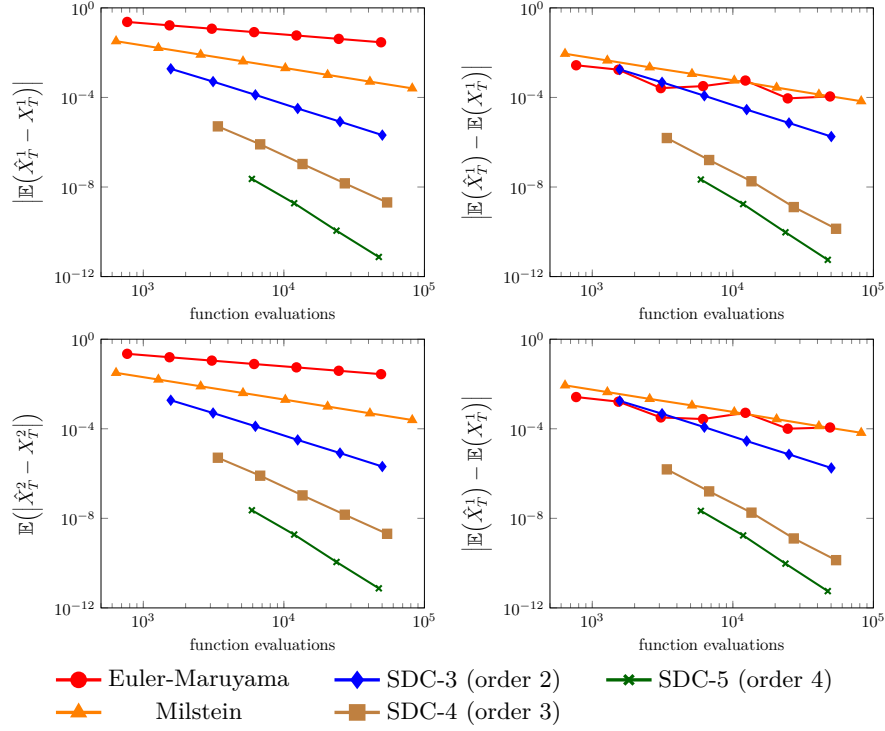

\subsubsection{Beyond linear bridge approximation.} Finally, we consider the Ornstein-Uhlenbeck process on $[0,1]$,
\begin{equation}
dX_t = \lambda\left(\mu-X_t\right)dt + \sigma \circ dW_t, X_0 = \frac{1}{2},
\label{ex:OU}
\end{equation}
with $\lambda = 2, \mu=1, \sigma=0.5$, and exact solution
\begin{equation}
X_t = \mu + (X_0-\mu)\exp(-\lambda t	) + \sigma \int_0^t \exp(-\lambda (t-s) )\ dW_s.
\end{equation}
As this SDE is not of the form~\eqref{eq:ReducSDE}, using only the linear part of the SBB approximation is not enough. To avoid falling back to order \textonehalf, we do not use spectral quadrature for integrating 
\begin{equation}
b(t,X)\biggl(\sum_{j=1}^{J}\cos\bigl(\frac{j\pi t}{\Delta T}\bigr)\eta_j\biggr),
\end{equation}
but note that for SDC $b(t,X)$ is approximated by a polynomial $p(t)$ given through its coefficients at the collocation nodes. As for such a polynomial the product $p(t)\cos\bigl(\frac{j\pi t}{\Delta T}\bigr)$ can by integrated exactly using repeated application of integration-by-parts, we precompute integration matrices for $\ell(t) \cos\bigl(\frac{j\pi t}{\Delta T}\bigr)$ with Lagrange polynomials $\ell(t)$, and apply these to evaluate the respective integrals during SDC sweeps. To separate error contributions, for this experiment we first approximated the Wiener process on a very fine temporal grid with time-step $2^{-17}$ globally by a Karhunen-Lo\'eve expansion with 64 modes, and used this smooth version to construct the smooth Brownian bridges $B^{(128)}$ with $J=128$ modes on the coarse time intervals $\Delta T \in \{ 2^{-2}, \dots, 2^{-7}\}$ used in the convergence study. We fix the number of collocation nodes to $4$, and use $10^5$ samples to estimate expected values. Resulting strong errors are shown in Fig.~\ref{fig:OUconv}. 

\begin{figure}
\centering
\includegraphics{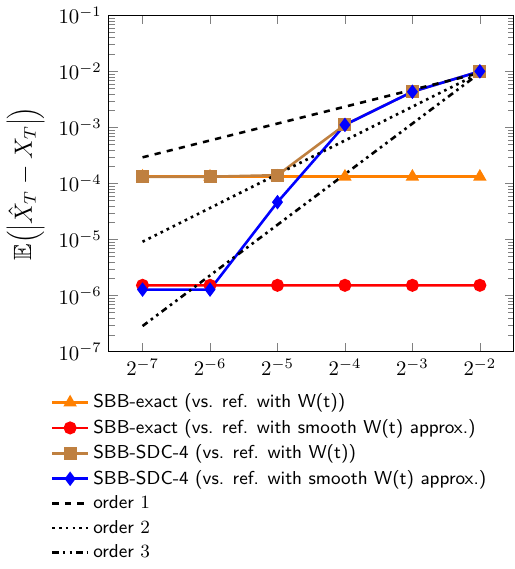}
\caption{Strong convergence for the Ornstein-Uhlenbeck example.}
\label{fig:OUconv}
\end{figure}

The figure illustrates two separate error contributions: For the SBB construction above, we solve the SBB-ODE exactly on the coarse time steps and compare to the exact solution obtained with the sampled Brownian motion $W(t)$ (orange-triangle line), and the smooth approximation of $W(t)$ on the full interval $[0, T]$ (red-circle line). The error of the SBB approximation with $128$ terms compared to the solution with the true Wiener process is $1.4\cdot 10^{-4}$, where the SBB approximation of the smooth process is significantly more accurate ($1.5\cdot 10^{-6})$. Regarding the SDC-SBB solution, we observe an experimental order $2.9$ determined by a linear fit to the error curve, close to the expected order $3$ when using the smooth reference (blue-diamond line), until it levels of at the SBB approximation error. Using the solution with the true nonsmooth Brownian motion as a reference (brown-square line), we observe order $2$ until the error level of the exact solution is reached, indicating that higher order can be reached with this approach, even though a reduction compared to the expected order $3$ occurs. However, such an order reduction is not unexpected for a stiff equation, where the fast oscillations cannot be resolved. Potential remedies include using implicit SDC, and will be investigated in future work.

\section{Conclusion}
\label{section:Concl}

We introduced an SDC-based method for solving SDEs given in Stratonovich form using a smooth Brownian bridge approximation, and give a convergence result for the simplied case of a linear approximation of the Brownian bridge. Numerical experiments validate the theoretical results, and show promising computational efficiency. However, the restriction of using only the linear term in the SBB approximation is a limitation that will need to be addressed. While first tests show the potential of achieving higher-order also for using more terms in the SBB approximation, this requires closer investigation. Here, future work will examine the influence of specialized quadrature on the convergence orders, adapt implicit or implicit-explicit SDC for the SBB-ODE, as well as develop problem-tailored solution representations, and will provide guidelines on the choice of algorithmic parameters to make the SDC-SBB method applicable for a wider range of problems. 

\subsection*{Acknowledgments}
We thank J. C. Mattingly (Duke University), M. Weiser (Zuse Institute Berlin), and T. J. Sullivan (University of Warwick) for numerous helpful discussions. L.B. further thanks S. Weißmann (University of Mannheim), for valuable discussions and additional input.

\printbibliography

\end{document}